\documentclass[reqno,a4paper,12pt]{amsart}

    \usepackage{amsmath,amscd,amsfonts,amssymb}
    \usepackage{mathrsfs,dsfont}
    \usepackage{color}
    \usepackage{mathtools}
    \usepackage{tikz-cd}
\usepackage{hyperref}
    \usepackage{subcaption}
\usepackage{setspace}
\usepackage{mathrsfs,dsfont}
\numberwithin{equation}{section}

   \usepackage[twoside=false]{geometry}

    \usepackage{tikz}
   \usetikzlibrary{decorations.pathreplacing}

    \numberwithin{equation}{section}
\numberwithin{figure}{section}

    \def\cW{\mathcal{W}}
    \def\cS{\mathcal{S}}

    \def\R{\mathbb{R}}
    
    \def\Q{\mathbb{Q}}
    \def\Z{\mathbb{Z}}
    \def\N{\mathbb{N}}

    \def\one{\mathds{1}}
    
    \newcommand{\Mod}[1]{\ (\mathrm{mod}\ #1)}

    \renewcommand\leq{\leqslant}
    \renewcommand\geq{\geqslant}

    \renewcommand{\ae}{{\mathrm{a.e.}}}
    \newcommand{\as}{{\mathrm{a.s.}}}
    \newcommand\X{\mathcal{X}}
    
    \newcommand{\ft}[1]{\widehat #1}

    \newcommand{\supp}{\operatorname{supp}}
    
    \newcommand{\diam}{\operatorname{diam}}

    \newcommand{\Tile}{\operatorname{Tile}}
    \newcommand{\Solution}{\operatorname{Coloring}}

    \newcommand{\zft}[1]{\mathcal{Z}(\ft{\one}_{#1})}

    \theoremstyle{plain}
    \newtheorem{thm}{Theorem}[section]
    \newtheorem{theorem}[thm]{Theorem}
    
    \newtheorem{lemma}[thm]{Lemma}

    \newtheorem{question}[thm]{Question}
    
    \newtheorem{conjecture}[thm]{Conjecture}
    
    \newtheorem*{claim*}{Claim}

    \theoremstyle{definition}
    \newtheorem{definition}[thm]{Definition}
    \newtheorem*{definition*}{Definition}
    \newtheorem*{remarks*}{Remarks}
    \newtheorem*{remark*}{Remark}
    \newtheorem{remark}[thm]{Remark}
    \newtheorem{example}[thm]{Example}

\begin{document}

	\title{Translational tilings: structured or wild?}

	\author{Rachel Greenfeld}
	\address{Department of Mathematics, Northwestern University, Evanston, IL 60208.}
	\email{rgreenfeld@northwestern.edu}

	\subjclass[52C22, 52C23, 03B25, 05B45, 37B52, 39A23]{52C22, 52C23, 03B25, 05B45, 37B52, 39A23}
	\date{}
	

\begin{abstract} 
The study of the structure of translational tilings has captivated mathematicians, scientists, and the general public for centuries and continues to thrive at the crossroads of analysis, combinatorics, dynamics, logic, number theory, and geometry. This vibrant field seeks to uncover the delicate divide between rigid structures and unpredictable, ``wild'' behaviors that arise when sets fill space by translations without gaps or overlaps. We provide an overview of this study and recent developments, highlighting its multidisciplinary nature and offering a glimpse into the process behind the results.
\end{abstract}

 \maketitle

\section{Introduction}

\emph{Tiling} a space means covering it without overlaps using a small number of different types of building blocks (often just one), which we refer to as \emph{tiles}. 
A simple example is a kitchen floor tiled with identical squares, neatly arranged next to each other. But tilings need not be so plain--they can be playful, like Escher’s lizard tiling, or even more surprising, with no patterns repeating, as in the aperiodic Penrose kite and dart or the recently discovered aperiodic hat tiling.

\begin{figure}[ht]%
    \centering
    \subfloat[\centering Square tiling.]{{\includegraphics[width=3.5cm]{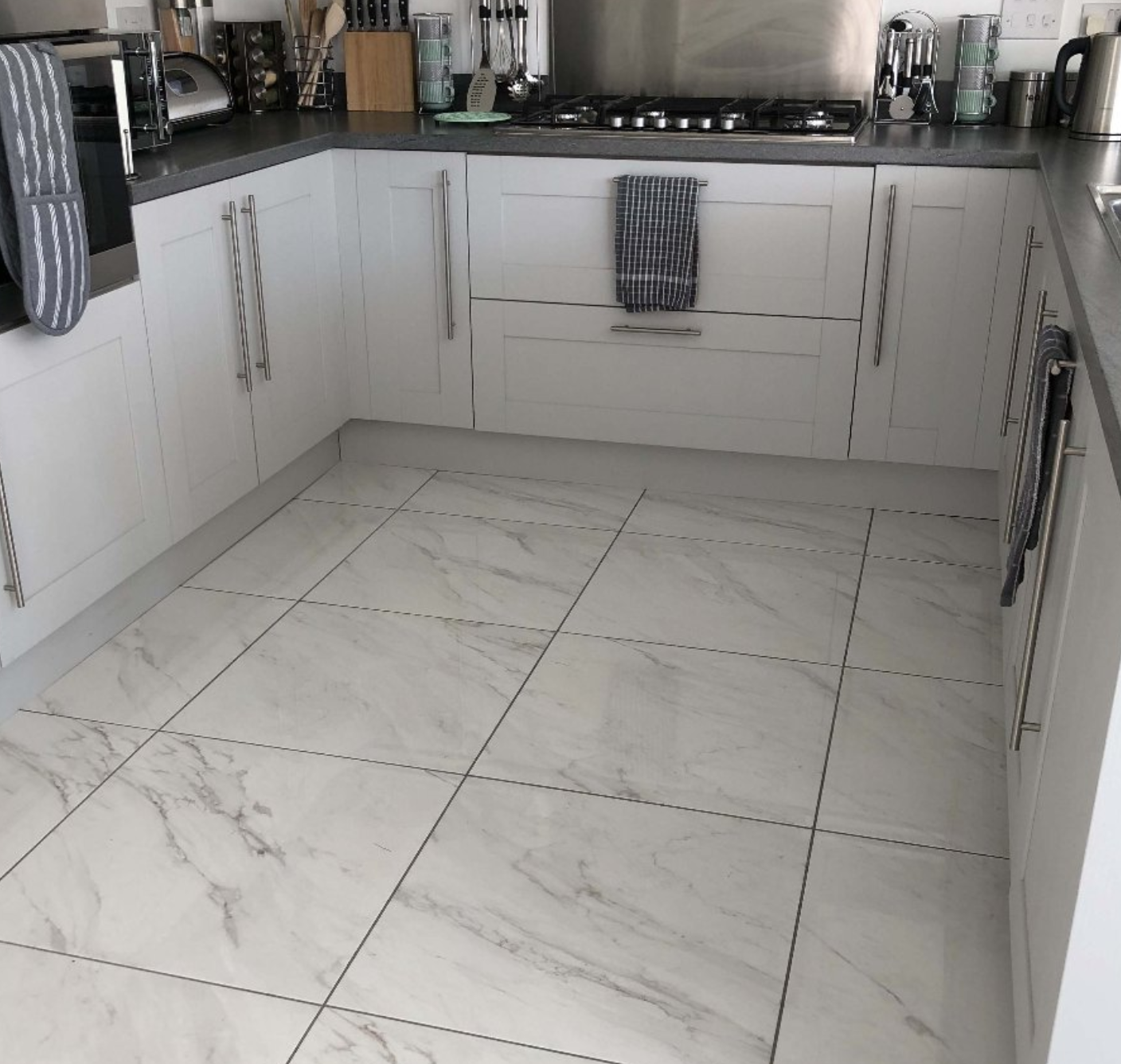} }}%
    \quad
    \subfloat[\centering Escher's lizard tiling.]{{\includegraphics[width=3.5cm]{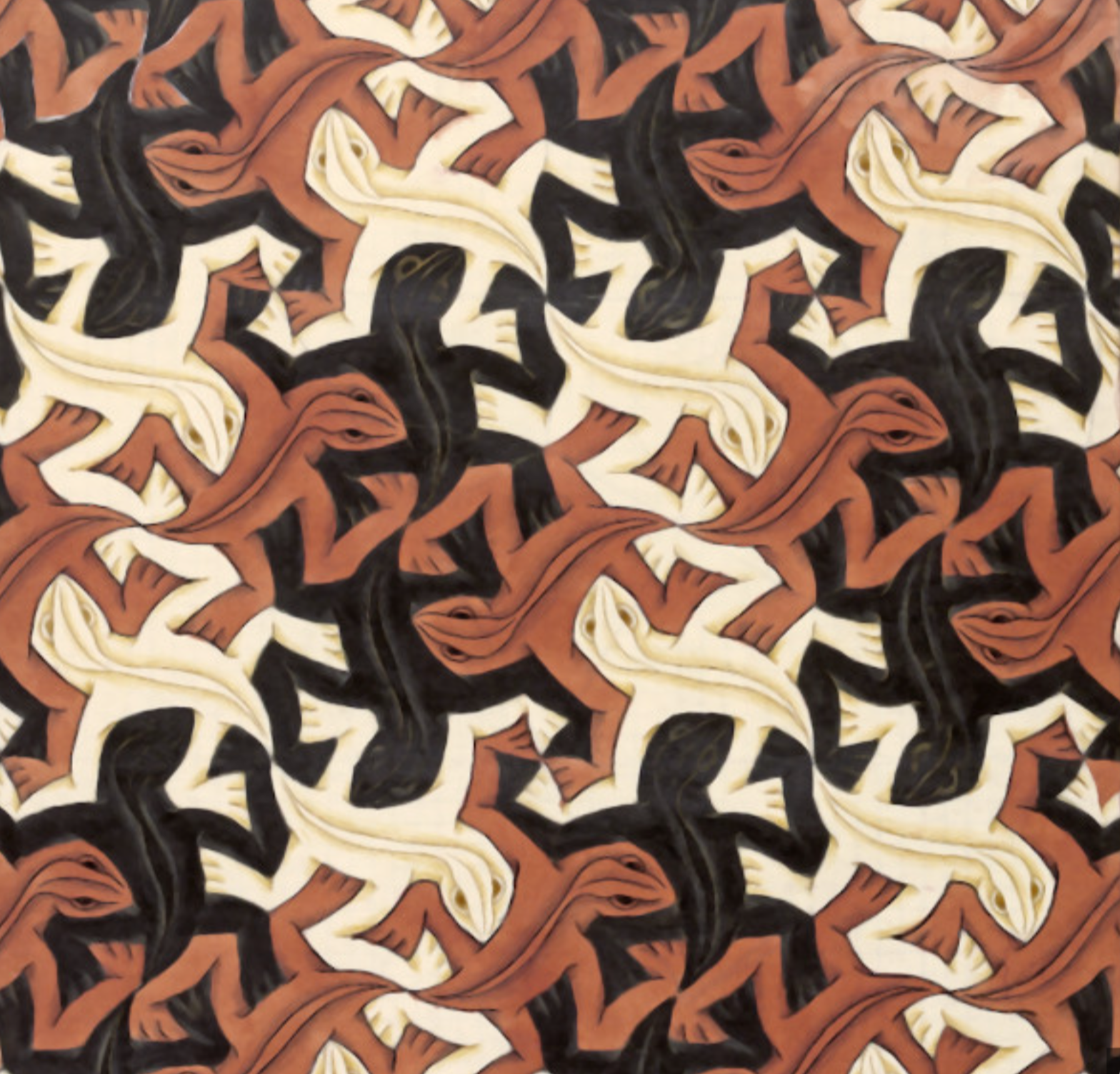} }}%
    \quad
    \subfloat[\centering Penrose kite-dart tiling.]{{\includegraphics[width=3.5cm]{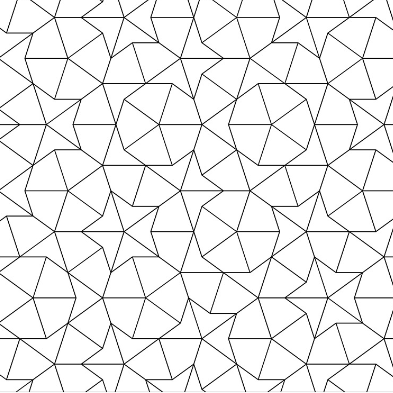} }}%
    \quad
    \subfloat[\centering Hat tiling \cite{hat}.]{{\includegraphics[width=3.5cm]{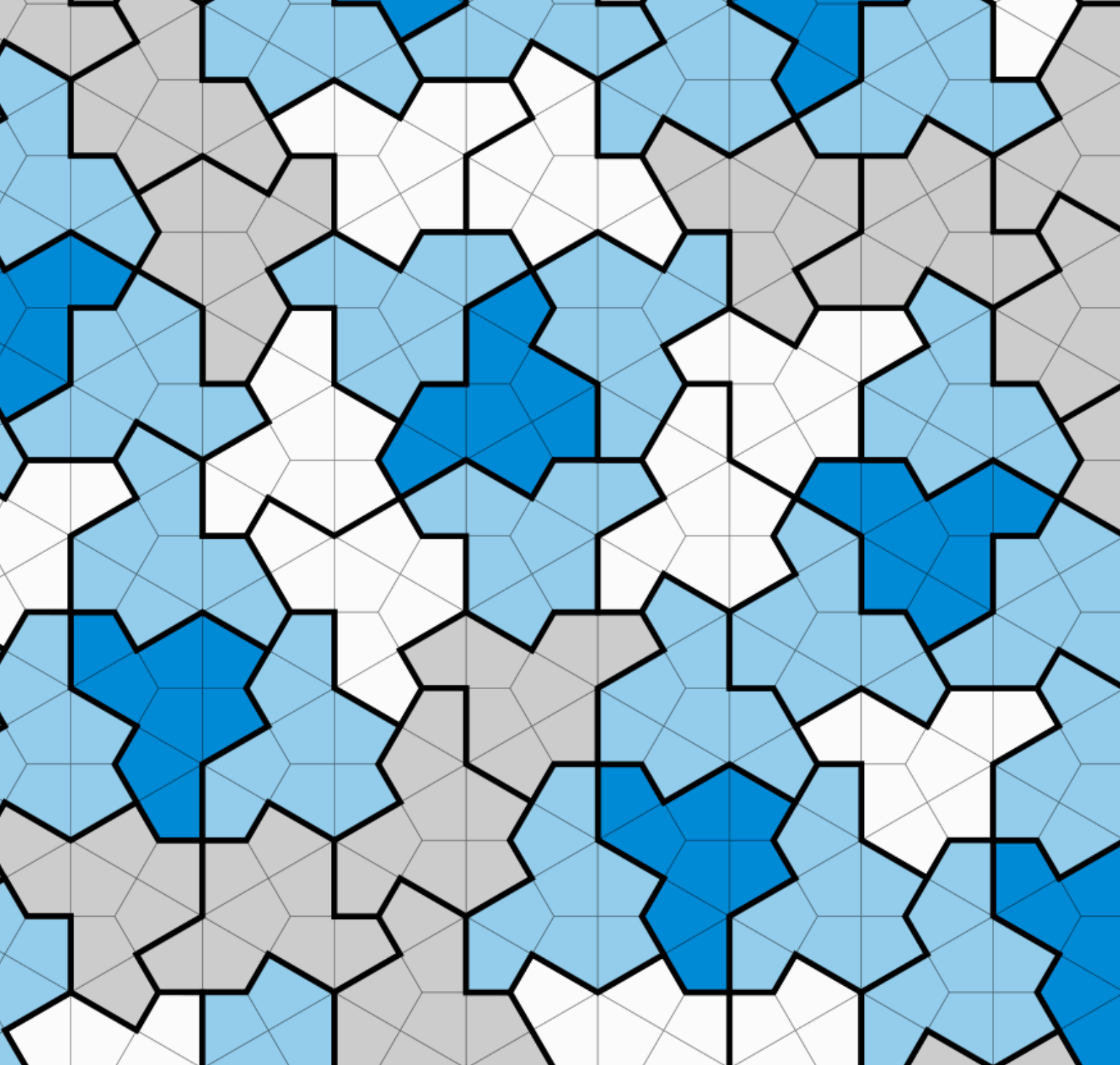} }}%
\end{figure}

Beyond aesthetics and tile decoration, these examples reveal a crucial distinction in their underlying geometry: unlike the square tiling, which consists solely of translated copies of the tile, creating an Escher lizard tiling requires both translations and rotations of the lizard tile, and the hat tiling requires not only translations and rotations, but also reflections of the tile to perfectly cover the plane (cf. the \emph{spectres} in \cite{SmithMyersKaplanGoodmanStrauss2024chiral}, which tile the plane aperiodically by translations and rotations only). Moreover, note that Penrose's tiling uses translations and rotations of \emph{two} tiles--kite and dart, rather than a single tile.
In this article, we exclusively explore the first, most basic case: tiling achieved through translations alone.

The study of tiling by translations has connections to parts of analysis, combinatorics, dynamics, algebra and number theory, as well as computability theory and, of course, geometry. It goes back to antiquity and has remained vibrant in the present through Hilbert's eighteenth problem \cite{hilbert1902problems}, Keller's conjecture \cite{Keller}, Wang's domino problem \cite{Wang1961,wang}, the periodic tiling conjecture \cite{GS,S74,LW}, Fuglede's spectral-set conjecture \cite{fuglede1974commuting}, the Coven--Meyerowitz conjecture \cite{coven1999tiling}, and other significant problems and conjectures.  
The restriction of a tile's allowed isometries to translations alone enforces a more rigid and structured behavior on the resulting tilings. Translational tilings (in Abelian groups) often have properties that more general tilings do not have.
However, it has turned out that these properties cease to hold when the dimension is sufficiently large. It appears that the extra freedom afforded by high dimension compensates for the rigidity imposed by restricting to translations.

 More generally, there seems to be a mysterious divide between structured tiling problems, in which the tilings are well behaved, and wild tiling problems, where almost anything can happen; e.g., 
 translational tilings  vs. isometric tilings \cite{gbn,hat}, monotilings  vs. tilings by multiple tiles \cite{GT,Ber,GT2}, low dimensional tilings  vs. high dimensional tilings \cite{N,B,GT22,GT23}, level one tilings  vs. higher level tilings \cite[Theorem 1.3]{GT}, etc.
We provide an overview of the study of this mystery, including recent developments (to date) in this area, carried out by the authors and collaborators, focusing mostly on a project joint with Terence Tao. 

We reveal some of the milestones in our study process--omitted from our published works yet crucial in the journey that led to our recent developments: how our thinking evolved, how failed attempts redirected us, and how these ultimately paved the way to our solutions.
Our main goal is to provide a window into our process, enabling other researchers to examine it and, we hope, build upon it to advance the field further.
Another objective is to convey the multidisciplinary aspect of our study process: how shifting perspectives, adopting different viewpoints, and combining tools from analysis, commutative algebra, combinatorics, dynamics, and logic eventually opened the gate to solving the periodic tiling conjecture.

\subsection{Notation and preliminary definitions.}
We use $|S|$ to denote the counting measure (the size) of a set $S$ in a finitely generated Abelian group; for a real number $r$, we use  $|r|$ to denote its magnitude. We use $\lfloor \cdot\rfloor\colon \R\to \Z$ to denote the integer part of a real number (i.e., the greatest integer less than or equal to that number), and $\{\cdot\}\colon \R\to [0,1)$ for the fractional part of a real number, i.e., $\{r\}\coloneqq r-\lfloor r\rfloor$. 

For sets $A,B$ we write $A\sqcup B$ to denote their union if these sets are disjoint and leave $A\sqcup B$ undefined if these sets are not disjoint. When $A,B$ are subsets of a group, we use the Minkowski sum notation:
$$A+B =\{a+b\colon a\in A,\; b\in B\};$$ clearly $A+B=B+A$ when the group is Abelian. We use the direct-sum notation $A\oplus B$ to denote the set $A+B$ if (almost) all the sums $a+b$, $a\in A, b\in B$ are distinct, and leave $A \oplus B$ undefined otherwise.
We also use the notation $\one_A$ for the indicator function of the set $A$, i.e., $$\one_A(x)=
\begin{cases}
    1 & \text{if } x\in A\\
    0 & \text{if } x\not\in A
\end{cases}.$$ 

We say that two nonzero vectors $v_1,v_2 \in \Z^d$, $d\geq 2$ are \emph{incommensurable} if they are linearly independent over $\Q$. We say that $n$ nonzero vectors $v_1,\dots,v_n \in \Z^d$ are \emph{mutually incommensurable} if $v_j,v_{j'}$ are incommensurable for all $1\leq j < j' \leq n$.

We say that a function $g\colon \Z\to \R/\Z$  \emph{equidistibutes} if 
$$\lim_{N\to \infty} \frac{1}{2N+1}\sum_{n=-N}^{N} \phi(g(n))=\int_{\R/\Z} \phi(t)dt$$
for every continuous function $\phi\colon \R/\Z \to \R$.

Let $\Sigma$ be a finite set. A set $\X\subset \Sigma^\Z$ is called a \emph{subshift} if it is compact in the product topology on $\Sigma^\Z$ and translation invariant (that is, $(x_n)_{n\in \Z}\in \X$ if and only if $(x_{n+1})_{n\in \Z}\in \X$).

\section{The domino problem: encoding wild behavior}\label{sec:domino}

In the early 1960s, H. Wang introduced a fundamental challenge known as \emph{Wang's domino problem} \cite{Wang1961,wang}. This problem involves a finite set of square tiles, each with colored sides, also known as \emph{Wang squares}. The goal is to determine whether an infinite plane can be completely covered by these tiles where the colors of all adjacent sides match.  A successful arrangement of the entire plane is called \emph{Wang tiling}; see Figure \ref{fig:wang} for an example.

\begin{figure}[ht]
    \centering
    \rotatebox{90}{\includegraphics[width=5cm]{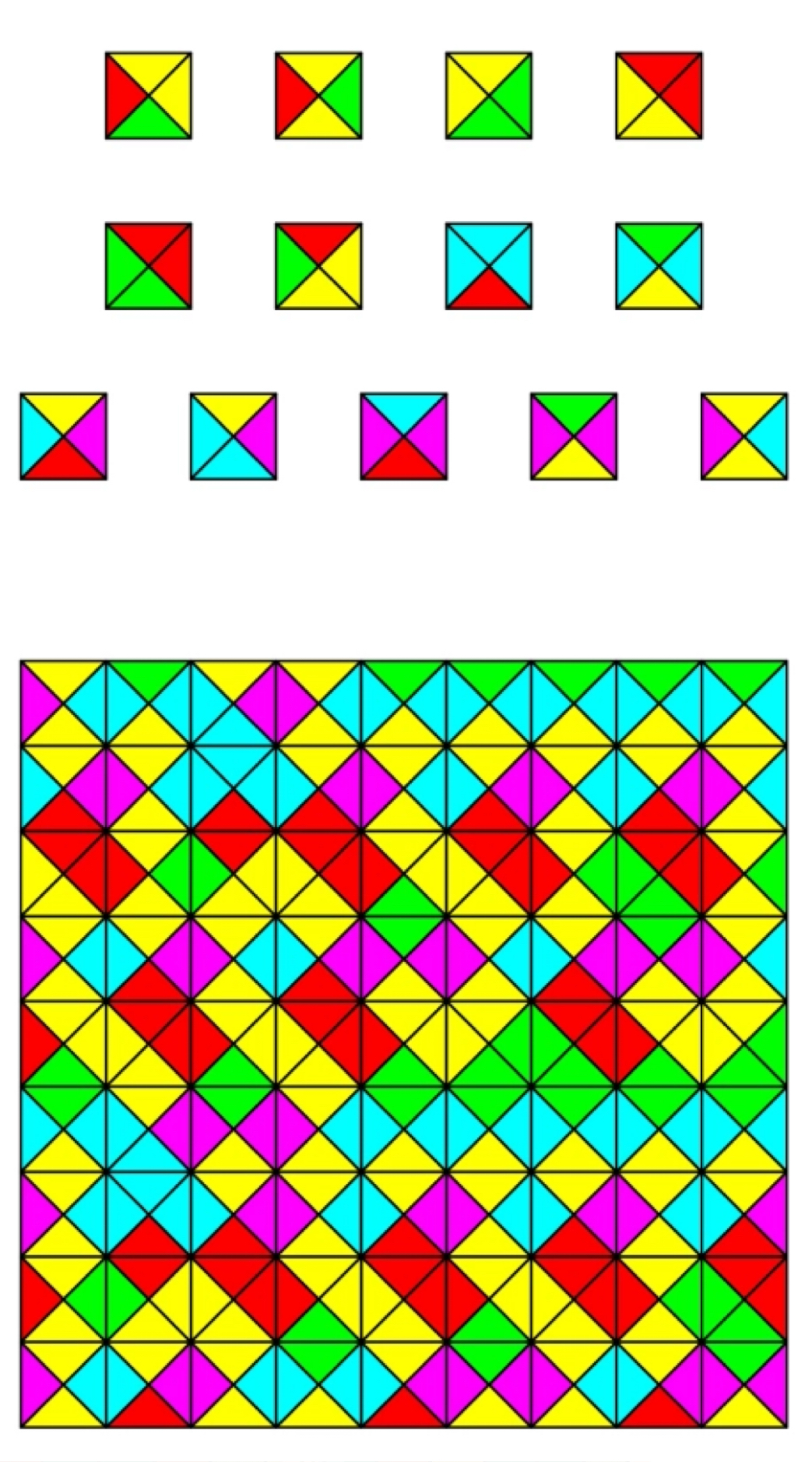}}
    \caption{On the left, a set of 13 Wang square tiles (over five colors). On the right, a piece of a Wang tiling with the given squares.}
    \label{fig:wang}
\end{figure}

Wang's initial interest in this geometric problem was driven by his work on the decidability of first-order logic $\forall\exists\forall$ formulas. He found that these formulas could be interpreted as domino problems, which raised a central question: Does there exist an algorithm that can, for an arbitrary set of Wang squares, determine in a finite number of steps if they admit a Wang tiling? 

Wang observed that the domino problem would be algorithmically decidable\footnote{This computability-theoretic concept of decidability is often used in the tiling literature. There is a closely related concept of  \emph{logical decidability} (also called \emph{provability in ZFC}), which refers to when a first-order sentence can be proved within the axiom system of ZFC. Logical decidability is stronger than algorithmic decidability: if a sentence is logically decidable, then it is also algorithmically decidable. Thus, a proof of the algorithmic undecidability of tilings immediately implies their logical undecidability.
See \cite[Section 1.1]{GT2} for more details.} if one could show that any finite set of Wang tiles that tiles the plane also has a tiling that is \emph{periodic}--repeating under translations by two independent vectors.
The idea is that an algorithm, given any finite set of Wang squares, could systematically attempt to tile increasingly large finite regions. If this domino problem is not solvable (i.e., the set does not admit a Wang tiling), compactness guarantees that the obstacle will be detected after finitely many steps. If the set does admit a Wang tiling, then by assumption it admits a periodic tiling; hence, after finitely many steps the algorithm will detect a finite region satisfying periodic boundary conditions and can be tiled by the given squares, guaranteeing that the given set tiles the plane. This gave rise to the study of the \emph{periodicity of tilings}.

A major breakthrough came with the proof that the fixed Wang tiling problem--determining if a finite set of Wang squares can create a tiling that includes a specific tile--is undecidable \cite{wang, buchi, kmw}. Building on this, Berger's landmark work \cite{Ber, Ber-thesis} provided the proof for the undecidability of the general Wang domino problem. Berger's proof demonstrated that the domino problem\footnote{A finite version of the domino problem is indicated as NP complete in \cite{GJ}.} is \emph{Turing complete}: by exploiting the hierarchical structure of \emph{substitution} tilings (filling of space by repeatedly inflating tiles and subdividing them according to a fixed rule), he showed that any Turing machine can be modeled as a Wang domino problem. Specifically, he showed that a given Turing machine halts if and only if its corresponding domino problem has no solution. This connection proved the undecidability of the Wang domino problem by reducing it to the undecidable halting problem \cite{turing}. Comprehensive surveys on the study of the undecidability of Wang's domino problem can be found in \cite{jv} and \cite[Section 1.1]{jr}.

\subsection{Translational tilings with multiple tiles.}
As noted by Golomb \cite{golomb}, a Wang tiling can be viewed as a specific type of translational tiling on the two-dimensional integer lattice $\Z^2$. Consequently, Berger’s proof implies that translational tilings with multiple tiles in $\Z^2$ are also undecidable. This undecidability reflects the wild behavior of multi-tile tilings, demonstrating an inherent and unpredictable complexity that defies systematic computation.

Although Berger’s work did not resolve the problem for a fixed, small number of tiles, his encoding approach paved the way for numerous advances in tiling theory, leading to undecidability results for tilings with fewer and fewer tiles; see, e.g., \cite{robinson,ollinger09,YZ,CZ,Yang24}.

 \subsection{Tiling language.}\label{subsec:language}
 In \cite{GT2}, we developed a \emph{tiling language} that allowed us to establish the undecidability of translational tilings with as few as two tiles.
Heuristically, the guiding idea was to reduce the number of tiles at the expense of increasing the dimension of the group.
This language enables us to express certain first-order sentences as systems of \emph{tiling equations} with $J$ tiles in a finitely generated Abelian group $G$, where both the group $G$ and the number of tiles $J$ remain fixed across all equations in the system (see \cite[Section 1.4]{GT2} for the definition of tiling equations with multiple tiles; and \cite[Example 1.14]{GT2}, \cite[Sections 4--5]{GT22} for tiling encoding examples). A crucial observation is that any such system of tiling equations can be expressed as a single tiling equation with the same number of tiles $J$ in a group whose dimension is increased by one \cite[Theorem 1.15]{GT2}.
We then create a library of constraints, similar to a programmer's collection of subroutines, expressible in our tiling language. This library is then be used to encode any domino problem as a tiling with only two tiles.

Note that a domino problem consists of a set of finitely many Wang squares $\cW\subset C^4$, where $C$ denotes the finite set of colors used in $\cW$. A Wang tiling with $\cW$ can be modeled as a function $a=(a_{\text{east}},a_{\text{south}},a_{\text{west}},a_{\text{north}})\colon \Z^2\to C^4$ such that for every $(n,m)\in \Z^2$ 
\begin{equation}\label{eq:domino1}
    a(n,m)\in \cW
\end{equation} 
and 
\begin{equation}\label{eq:domino2}
    a_{\text{east}}(n,m)=a_{\text{west}}(n+1,m) \; \text{ and } \; a_{\text{north}}(n,m)=a_{\text{south}}(n,m+1).
\end{equation}
We show that these sentences can be expressed in our tiling language as a tiling equation in $\Z^2\times G_0$ with two tiles (for a suitable finite Abelian group $G_0$). Thus, these two tiles admit a tiling of $\Z^2\times G_0$ if and only if the encoded domino problem is solvable. As the solvability of the Wang domino problem is undecidable, we obtain the undecidability of translational tilings with only two tiles in \emph{virtually-$\Z^2$ spaces}, i.e., spaces of the form $\Z^2\times G_0$ for arbitrary finite Abelian group $G_0$ \cite[Theorem 1.8]{GT2}. By pulling back this result (see \cite[Section 9]{GT2} for details), we then deduce the undecidability of translational tilings with two tiles in $\Z^d$ \cite[Theorem 1.9]{GT2}.

\begin{remark}
In our undecidability result with two tiles in spaces $\Z^2\times G_0$, neither the group $G_0$ nor its rank is fixed, but given as part of the input (along with a finite subset $F\subset \Z^2\times G_0$). In turn, the undecidability result for translational tilings with two tiles in $\Z^d$ we obtain is with the dimension $d$ not being fixed, but provided as part of the input.\footnote{While writing this article, a new preprint appeared on arXiv \cite{Kim25}, claiming to establish the undecidability of translational tilings with two polycube tiles in $\Z^3$.} Moreover, by rigidly inflating the latter result one can extend the undecidability of translational tilings with two tiles to the Euclidean space $\R^d$, where, again, the dimension $d$ is not fixed but is provided as part of the input.
\end{remark}

We see that wild behavior occurs even when restricting the number of tiles to as few as two. 
A natural follow-up question is the decidability of translational tilings with a \emph{single} tile, also known as \emph{translational monotilings}.
As discussed in \cite[Section 10]{GT2}, translational monotilings are substantially different from tilings with several tiles.  

\begin{remark}\label{rem:obstacle}
    In fact, the sentence \eqref{eq:domino1} could be encoded as a tiling equation in virtually-$\Z^2$ space with only one tile. It is the second sentence \eqref{eq:domino2}--the domino matching rules, which requires an additional tile to be encoded due to the simultaneous change in both the coordinate (east to west or south to north) \emph{and} the point ($(n,m)$ to $(n+1,m)$ or to $(n,m+1)$). If, however, we allow the finite group $G_0$ to be \emph{non-Abelian}, then we can encode both sentences as a tiling with only one tile in $\Z^2 \times G_0$ \cite[Theorem 11.2]{GT2}.
\end{remark}

In what follows, we dive into the theory of translational monotilings, and explore the strengths and weaknesses of this theory.

\section{Discrete translational monotilings}\label{sec:discrete}

Let $G=(G,+)$ be a finitely generated Abelian group, endowed with the discrete topology. Thus, $G$ is isomorphic to $\Z^d\times G_0$ for some $d\geq 0$ and a finite Abelian group $G_0$.  We say that a finite set $F\subset G$ \emph{tiles $G$ by translations} if there exists a set $A\subset G$ such that the translates 
\begin{equation}\label{eq:Ftranslates}
    F + a=\{f+a\colon f\in F\},\quad a\in A
\end{equation} of $F$ along $A$ form a \emph{partition} of $G$, i.e.,  for every point $g$ in $G$ there exists a \emph{unique} pair $(f,a)$ in $F\times A$ such that $g = f +a$. We refer to such $F$ as a \emph{tile} of $G$ and $A$ as a \emph{tiling of $G$ by $F$}, and say that the \emph{tiling equation} $F\oplus A=G$ holds.
In other words, $F$ tiles $G$ by translations along $A$ whenever the following two conditions are satisfied:
\begin{itemize}
    \item {\it (packing)} Each point in $F+A$ is covered exactly once by the sets \eqref{eq:Ftranslates}, meaning that $F\oplus A$ is well defined. 
    \item {\it (covering)} $F+A=G$.
\end{itemize}

We consider the \emph{indeterminate tiling equation}: \begin{equation}\label{eq:teq}
	    F \oplus X=G,
	\end{equation}
 where we view the finite set $F\subset G$ as given data and the indeterminate variable $X$ denotes an \emph{unknown} subset of $G$. We define the \emph{solution space} of the tiling equation \ref{eq:teq} as
 \begin{equation}\label{sol}
     \Tile(F;G)\coloneqq\{A\subset G\colon  F \oplus A= G\}.
 \end{equation} 
 Then, the given set $F$ tiles $G$ by translations if and only if the solution space \eqref{sol} is nonempty.

\begin{example}[Singleton tile]\label{ex1}
    Any singleton $\{x\}\subset G$ is a translational tile as $\{x\}\oplus G = G$. Indeed, when translating $F$ by every element in $A\coloneqq G$, every point of $G$ is covered exactly once. 
    In fact, this is the only tiling of $G$ by $\{x\}$, i.e.,  $\Tile(\{x\} ; G)=\{G\}$. This is the simplest form of a translational tiling.
\end{example} 
\begin{example}[Non-tiling example]
     The set $\{0,2,3\}$ is not a translational tile of $\Z$, as the set $\Tile(\{0,2,3\};\Z)$ is empty. Indeed, suppose $A\subset \Z$ were a tiling in $\Tile(\{0,2,3\};\Z)$; by translation invariance, we can assume that $0\in A$. Then, by the covering assumption, there is $a$ in $A\setminus \{0\}$ such that $1\in \{0,2,3\} +a$; hence $a\in \{-2,-1,1\}$. But this implies an overlap between the copies of $\{0,2,3\}$ at $a$ and $0$, contradicting the packing assumption. 
     This example demonstrates that not all finite sets can act as translational tiles.
\end{example}
\begin{example}[Discrete square tile]\label{ex3}
The discrete square $F=\{0,1\}^2$ tiles $\Z^2$ along 
    $$A_a\coloneqq \{(2n,2m+a(n))\colon n,m\in \Z\} \quad \text{as well as} \quad A^a\coloneqq \{(2n+a(m),2m)\colon n,m\in \Z\}$$ for every function $a\colon \Z\to \{0,1\}$. In the tiling $A_a$, the columns are sifted independently, while in $A^a$, the rows are sifted independently. See Figure \ref{fig:square}.
\end{example}

\begin{figure}[ht]
    \centering
    \includegraphics[width=10cm]{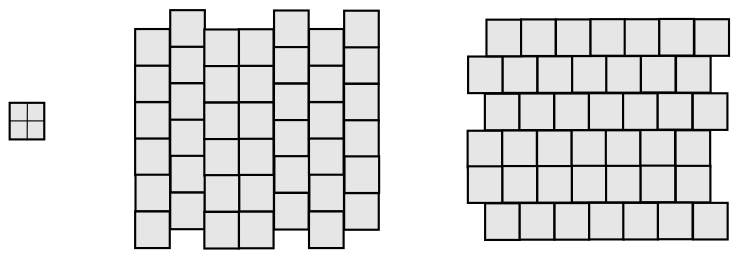}
    \caption{On the left, a discrete square tile $F=\{0,1\}^2$ (where a point $x$ in $\Z^2$ is illustrated by a continuous unit square $x+[0,1]^2$). In the middle, a piece of a tiling by $F$ of type $A_a$ with columns shifted independently, while on the right is a tiling $A^a$, whose rows are shifted independently.}
    \label{fig:square}
\end{figure}

\begin{example}[Disconnected tile]\label{ex4}
    For $F=\{0,2\}\times \{0,1\}$, we have that the set
    $$A_{a}^b\coloneqq \left\{(4n,2m+a(n))\colon n,m\in \Z\right\}\sqcup \left\{(4n+1+2b(m),2m)\colon n,m\in \Z\right\}$$ is in $\Tile(F; \Z^2)$ for every choice of functions $a,b\colon \Z\to \{0,1\}$. Here, there are independent shifts both in the vertical and horizontal directions. See Figure \ref{fig:disconnected}.
\end{example}

\begin{figure}[ht]
    \centering
    \includegraphics[width=8cm]{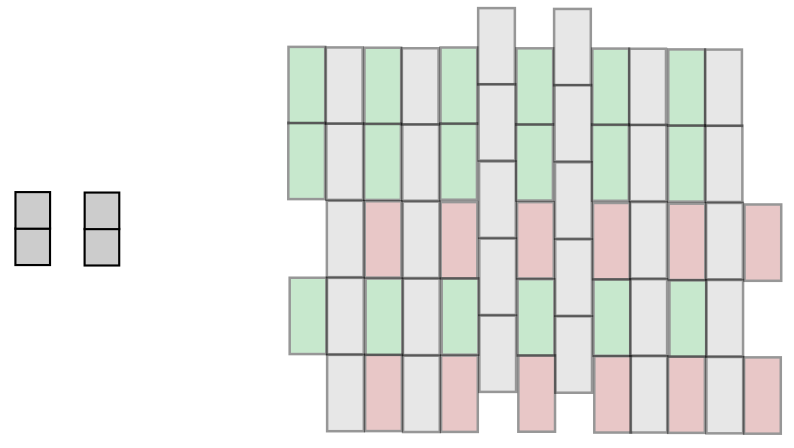}
    \caption{On the left, the discrete disconnected tile $F=\{0,2\}\times \{0,1\}$ (where a point $x$ in $\Z^2$ is illustrated by the continuous unit square $x+[0,1]^2$). On the right, a piece of a tiling $A_a^b$ by $F$ with both columns (grey) and rows (green and red) shifted independently.}
    \label{fig:disconnected}
\end{figure}

We are interested in the question of whether a given tiling equation \eqref{eq:teq} admits a solution, or, more generally: what does the solution space $\Tile(F;G)$ look like?

\subsection{Elementary properties.} Some structural properties of the space $\Tile(F;G)$ can be easily deduced from its definition.

\begin{itemize}
    \item \emph{Translation-invariance.} A fundamental property of tilings is that they are invariant under translation. 
It follows immediately from the definition of tiling that $\Tile(F;G)=\Tile((F-x);G)$ for every $x\in G$. In other words, $A$ is a tiling of $G$ by $F$ if and only if $A+x$ is a tiling of $G$ by $F$ for every $x\in G$. From a dynamical viewpoint, this property is useful. 
 
\item \emph{Density.} Every tiling $A$ in $\Tile(F;\Z^d)$ has a well-defined \emph{asymptotic density}, which is equal to $\frac{1}{|F|}$; i.e., for every $A\in \Tile(F;\Z^d)$ and  $x\in \Z^d$  the limit 
$$\lim_{N\to\infty} \frac{|(A-x)\cap \{-N,\dots,N\}^d|}{(2N+1)^d}$$
exists and is equal uniformly to $\frac{1}{|F|}$ (see, e.g., \cite[Lemma 1.2]{K} for details). The notion of density provides a quantitative measure of a set's distribution over its space--how spread out or clustered it is.
As such, this property imposes some degree of regularity on the structure of tilings by $F$.

\item \emph{Reflection-invariance.} By the commutativity of the group action, we have that $\Tile(F;G)=\Tile(-F;G)$, where $-F=\{-f\colon f\in F\}$  (see, e.g., \cite[Theorem 13]{szegedy} or the proof of \cite[Theorem 1.2(i)]{ggrt} for details). Since this property relies on the commutativity of the group action, it does not hold, e.g., when more tile isometries are allowed.
\end{itemize}

Translational monotilings have an additional algebraic structure that other tilings do not have, making them significantly more rigid. The rigid nature of translational monotilings gives rise to various structure problems.

\subsection{Different approaches and tools.}
Equipped with the topology of pointwise convergence, the space $\Tile(F;G)$ is compact and invariant under the action of the group $G$ by translation;
thus, it is a \emph{topological dynamical system}, which we can equip with a (Borel) probability measure to obtain a \emph{probability space}, or with an ergodic measure to obtain an \emph{ergodic system}. This gives rise to the use of probabilistic methods as well as ergodic theoretic methods to study the space $\Tile(F;G)$ (see, e.g., \cite{B,GT25}).

Also, note that the tiling $F \oplus A = G$ can be written in convolution form as 
$$\one_F * \one_A (x) =\sum_{f\in F} \one_A(x-f) =1, \quad \forall x\in G, $$
which looks tempting to analyze via the \emph{Fourier transform} (see \cite{K}, \cite[Section 2]{Kol} and the references therein). Taking distributional Fourier transform of $\one_F *\one_A=1$ (as $\one_A$ is not necessarily a measure but rather a tempered distribution), we obtain the equation $\ft{\one_F} \cdot \ft{\one_A} = \delta_0$ in frequency space (with $\delta_0$ being the Dirac distribution), which provides significant structural information about $\one_A$. In particular, we can deduce the necessary condition: 
\begin{equation}\label{eq:supp}
    \supp \ft{\one_A} \subset \{0\} \cup \zft{F}
\end{equation}
for $F\oplus A=G$, where $\zft{F}$ denotes the set of zeroes $\{\xi\in \widehat{G} \colon \ft{\one_F}(\xi)=0\}$ of the trigonometric polynomial $\zft{F}$ (this condition becomes sufficient as well when $\ft{\one_A}$ is locally a measure).  Thus, by analyzing the set $\ft{\one_F}$ we can deduce substantial structural information about the tiling $A$. 
For instance, by the continuity of the polynomial $\ft{\one_F}$ and since $\ft{\one_F}(0)=|F|$, the condition \eqref{eq:supp} implies a spectral gap in $\one_A$, which is a powerful piece of data to use when analyzing the structure of the tiling $A$ (see, e.g., \cite{B,LW,KL}).

\subsection{Dilation-invariance.} In addition to translation invariance, discrete translational tilings possess a unique dilation-invariance property, which is captured by the following \emph{dilation lemma}.

\begin{lemma}[Dilation lemma]\label{lem:dilation}
    Let $F$ be such that $\Tile(F;G)$ is nonempty. For sufficiently divisible\footnote{For instance, we can choose the number $q$ to be $p|F|$, where $p$ is the exponent of the finite Abelian group $G_0$ such that $G=\Z^d\times G_0$.} $q$ we have $$\Tile(F;G)\subset \bigcap_{r\Mod q=1} \Tile(rF;G),$$ where $rF=\{rf\colon f\in F\}$ for an integer $r$.
\end{lemma}

The proof of Lemma \ref{lem:dilation} is algebraic in nature and can be found in \cite[Theorem 10]{szegedy}, \cite[Corollary 11]{hk}, \cite[Proposition 3.1]{B}, \cite[Lemma 3.1]{GT},  and \cite[Lemma 3.1]{GT25}.  A one-dimensional version of this lemma previously appeared in \cite{tijdeman}. An analog for measurable tilings was proved in \cite[Theorem 1.2(i)]{ggrt}. See also \cite[Proposition 3.2]{hiprv} and \cite[Theorem 3.3]{imp} for a related Fourier-analytic dilation lemma for tilings in finite fields. Variants of this lemma for finitary configurations were also established in \cite[Lemma 3.2.2]{sz} and \cite[Lemma 7]{ks}. 

\begin{figure}[ht]
    \centering
    \includegraphics[width=7cm]{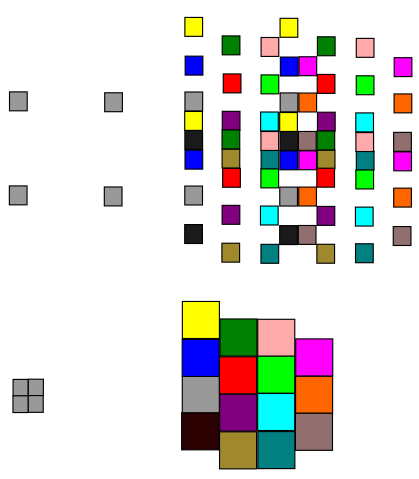}
    \caption{On the bottom, the square tile $F=\{0,1\}^2$ (where a point $x$ in $\Z^2$ is represented as $x+[0,1]^2$) and a piece of a tiling $A\in \Tile(F;\Z^2)$. For this $F$ we can let $q=4$. On the top is the dilation $rF=\{0,5\}^2$ of $F$ by the dilation factor $r=5$ with $r=1\Mod q$,  and we see that $A$ also form a tiling of $\Z^2$ by $rF$.}
    \label{fig:dilation}
\end{figure}

Lemma \ref{lem:dilation} plays a key role in the study of the structure of translational monotilings.
Observe that, informally, it implies long-term correlations\footnote{For example, it eliminates the existence of nontrivial \emph{factor of iid tilings} \cite[Theorem 1.5]{ggrt}.} between elements of a tiling $A$. Indeed, although we blow up the tile $F$, making it increasingly sparser by enlarging the dilation factor $|r|$, these blown up versions of $F$ still tile along $A$ (see Figure \ref{fig:dilation}). 
Informally, this shows that even far apart elements of $A$ can still ``feel'' each other in some way -- a reason to believe that the tiling $A$ exhibits some level of \emph{periodicity}.

\begin{remark}
    The \emph{commutative} nature of the group as well as the fact that there is only a \emph{single} tile are both crucial in the proof of the dilation lemma. Indeed, this lemma fails for non-Abelian group operations, or tilings by multiple tiles, and therefore it is significantly easier to establish wild behavior of such tilings, as described in Section \ref{sec:domino} and Remark \ref{rem:obstacle}.
\end{remark}

\subsection{Periodicity of monotilings.}

\begin{definition}[Periodicity and aperiodicity]
Let $G=\Z^d\times G_0$ be a finitely generated Abelian group,  $\Lambda$ be a subgroup of $G$ and $F$ be a finite subset of $G$.  
\begin{enumerate}
    \item[(i)] A {\bf function} $f$  on $G$ is \emph{$\Lambda$-periodic} if $f(x+\lambda)=f(x)$ for every $x\in G, \lambda\in \Lambda$. Moreover, if $\Lambda$ is a lattice (i.e., has finite index),  we say that $f$ is \emph{periodic}. (In the literature, this notion is sometimes referred to as \emph{strong periodicity} or \emph{full periodicity}.) If $f$ is not periodic, we say that it is \emph{non-periodic}.
    And, if $\Lambda$ is generated by a single vector, we say that $f$ is \emph{singly-periodic}.
\item[(ii)] A {\bf set} $E\subset G$ is \emph{$\Lambda$-periodic (resp. periodic, non-periodic, singly-periodic)} if the indicator function $\one_A$ is $\Lambda$-periodic (resp. periodic, non-periodic, singly-periodic). 
Moreover, when $d=2$, we say that $E$ is \emph{weakly-periodic} if it can be written as a finite disjoint union of singly-periodic sets. 

\item[(iii)]  A {\bf tiling} $A\in \Tile(F;G)$ is called \emph{$\Lambda$-periodic (resp. periodic, non-periodic, singly-periodic, weakly-periodic)} if the set $A\subset G$ is $\Lambda$-periodic (resp. periodic, non-periodic, singly-periodic, weakly-periodic).
\item[(iv)] A {\bf tiling equation} \eqref{eq:teq} is said to be \emph{aperiodic} if the solution space $\Tile(F;G)$ is nonempty and every tiling in $\Tile(F;G)$ is non-periodic.
\end{enumerate}
\end{definition}

To demonstrate these notions of periodicity, we can look back at the tiling examples above. 
 In Example \ref{ex1}, for instance, the tiling is a lattice and therefore is in particular periodic. 
 In Example \ref{ex3}, however, for most choices of $a\colon \Z\to \{0,1\}$, the tiling obtained will not be periodic. Nevertheless, all these tilings are singly-periodic. (In fact, in \cite{gbn} it was shown that any \emph{simply connected} tile $F\subset \Z^2$, in the sense that $F+[0,1]^2$ is simply connected in $\R^2$, admits only singly-periodic tilings.)
Moreover, by choosing a periodic function $a\colon \Z\to \{0,1\}$ we obtain a periodic tiling; therefore the tiling equation $\{0,1\}^2\oplus X=\Z^2$ is not aperiodic. 
 In Example \ref{ex4}, most choices of $a,b\colon \Z\to \{0,1\}$ will give rise to a tiling which is not even singly-periodic (hence, non-periodic). However, all of these tilings are weakly-periodic, and by choosing both functions $a$ and $b$ to be periodic we would obtain periodic tilings by $F=\{0,2\}\times \{0,1\}$. Thus, the corresponding tiling equation is not aperiodic.

\subsubsection{The periodic tiling conjecture.}
The \emph{periodic tiling conjecture} \cite{S74, GS, LW} is a paramount conjecture in the study of the structure of translational tilings. This conjecture asserts that if  $F$ tiles $G$ by translations, then it must admit at least one \textit{periodic} tiling.

\begin{conjecture}[Discrete periodic tiling conjecture]\label{con:ptc}  
 Let $G$ be a finitely generated Abelian group and $F$ be a finite subset of $G$.  Then the tiling equation $F \oplus X = G$ is \emph{not} aperiodic.
 \end{conjecture}

An illustration of the importance of the study of the periodic tiling conjecture is its tight connection with the question of the decidability of tilings. Recall that if any tiling admits a periodic solution, then the tiling problem is decidable, i.e., there is an algorithm which computes (in finite time), upon any input of $G$, $F\subset G$, whether the solution space $\Tile(F;G)$ is empty or not (see Section \ref{sec:domino} for a description of the algorithm). As discussed in Section \ref{sec:domino}, tilings with multiple tiles are known to be undecidable; and, in fact, the proof \cite{Ber,Ber-thesis} of this undecidability result consists of a construction of aperiodic substitution tilings with multiple tiles.
 
The study of the periodic tiling conjecture attracts many researchers and involves methods from many different areas such as geometry \cite{gbn,ken,err}, Fourier analysis \cite{LW}, combinatorics \cite{GT,GT2}, ergodic theory and probability \cite{B,GT25}, commutative algebra \cite{szegedy,hk,B,GT}, model theory \cite{GT2,GT23,GT25}, and computability theory \cite{Ber,GT23}.

The following partial results supporting Conjecture \ref{con:ptc} are known:

\begin{itemize}
    \item Conjecture \ref{con:ptc} is trivial when $G$ is a finite Abelian group, since in this case all subsets of $G$ are periodic.
\item  In \cite{N,Sidorenko}, Conjecture  \ref{con:ptc}  was established for $G=\Z$ using a simple pigeonhole-principle argument.  In fact, it was shown that \emph{every} tiling in $\Tile(F;\Z)$ is periodic. In \cite[Corollary 3.5]{GT}, we gave an alternative proof, relying on Lemma \ref{lem:dilation}, which improves the bound on the universal period and on the computational complexity in turn.
\item In \cite[Section 2]{GT2} we proved Conjecture \ref{con:ptc} in the virtually-$\Z$ case, i.e., for groups of the form $G = \Z \times G_0$, where $G_0$ is any finite Abelian group. The proof adapts the pigeonhole-principle argument of \cite{N} to this case.
\item For $G=\Z^2$, Conjecture \ref{con:ptc} was established by Bhattacharya \cite{B} using ergodic theory methods.  In \cite{GT} we gave an alternative (quantitative) proof of this result and, furthermore, showed that every tiling in $\Z^2$ by a single tile is \textit{weakly periodic}. Both arguments rely heavily on Lemma \ref{lem:dilation}.    
\item In \cite{szegedy}, Conjecture \ref{con:ptc} was proven in $\Z^d$ when the size $|F|$ of $F$ is prime or equal to $4$, using only commutative algebra and elementary number theory. 
\end{itemize}

   Despite these positive evidences for Conjecture \ref{con:ptc}, we recently established that the conjecture is false already for virtually-$\Z^2$ spaces by an explicit construction in a group $\Z^2\times G_0$ for some finite Abelian group $G_0$ \cite{GT22}. Using the reduction \cite[Corollary 1.2]{bgu}, we then conclude the failure of Conjecture \ref{con:ptc} in $\Z^d$ for sufficiently large $d$. Moreover, in \cite{GK}, we showed that Conjecture \ref{con:ptc} fails in sufficiently high dimensions even when restricted to \emph{connected tiles}, i.e., tiles $F\subset \Z^d$ such that $F+[0,1]^d$ is connected in $\R^d$.

In what follows, building on Lemma \ref{lem:dilation}, we introduce a structure theory for translational monotiling and describe how to exploit this theory to study Conjecture \ref{con:ptc}.

\section{Reversible structure theorem}
In this section, we deduce a reversible structure theorem for discrete translational monotilings from the Lemma \ref{lem:dilation}. This theorem plays a crucial role in our studies, since it completely captures the tiling property.

\begin{theorem}[Reversible structure theorem]\label{thm:structure}
Let $d\geq 1$, $G_0$ be a finite Abelian group, and $G=\Z^d\times G_0$ a finitely generated Abelian group. Suppose $F$ is a finite subset of $G$. For $x\in \Z^d$ let $F_x$ denote the slice $F\cap \{x\}\times G_0$ of $F$ and let $S_F\coloneqq \{x\in\Z^d \colon F_x\neq \emptyset\}$. Then there exist a sufficiently divisible integer $q$ and a finite set\footnote{We can take $V$ to be the set of directions of the (nonzero) vectors in $S_F-S_F$.} $V$ of primitive, nonzero, mutually incommensurable vectors in $\Z^d$ such that for every $A\subset G$ the following are equivalent:
\begin{enumerate}
    \item[(i)]  $F \oplus A = G$.
    \item[(ii)] For every $f\in S_F$ there is a decomposition 
    \begin{equation}\label{eq:decomp}
        \one_{F_f} * \one_A = 1 - \sum_{v \in V} \varphi_{f,v},
    \end{equation}  
    where for each $v\in V$, the function $\varphi_{f,v}: G \to [0,1]$ is a $\langle qv\rangle$-periodic (viewing $\Z^d$ as a subgroup of $G$) such that $\varphi_v\coloneqq \sum_{f\in S_F} \varphi_{f,v}$ is periodic and  
    \begin{equation}\label{eq:sumc}
       \sum_{v\in V}\varphi_{v} =|S_F|-1.
    \end{equation}
\end{enumerate} 
\end{theorem}

\begin{remark}\label{rem:G0trivial}
    Note that when $G_0$ is trivial, \eqref{eq:decomp} is a decomposition of the tiling $\one_A$ itself, and $|S_F|=|F|$ in this case. 
\end{remark}

\begin{proof}[Proof of Theorem \ref{thm:structure}]
    Suppose first that (i) holds. Then \eqref{eq:decomp} follows from \cite[Theorem 3.2]{GT25}, where the key ingredient in the argument is Lemma \ref{lem:dilation}.
    Moreover, by (i), applying the projection operator $\pi_{qv}$ to $\sum_{f\in S_F}\one_{F_f} * \one_A = 1$ (or, averaging along the direction $qv$, see \cite[Equation (2.3)]{GT25}) we obtain that the function $\varphi_v$ is periodic, and by \eqref{eq:decomp} we also have
    \begin{equation}\label{eq:sumSf}
        \sum_{f\in S_F} \one_{F_f}*\one_A= |S_F|-\sum_{v\in V}\varphi_v,
    \end{equation}
    giving (ii) as the left hand side is equal to the constant function $1$. 
    Conversely, if (ii) holds, then by \eqref{eq:decomp} we have \eqref{eq:sumSf}, which, by \eqref{eq:sumc} and the fact that $\one_F*\one_A =\sum_{f\in S_F} \one_{F_f}*\one_A$, implies (i). 
\end{proof}

\begin{remark} 
A related structural decomposition of tilings in $\Z^2$ was previously established in \cite{B}.
    In this paper, Bhattacharya introduces an ergodic theoretic approach to the study of the structure of translational tilings and Conjecture \ref{con:ptc}, by viewing $\Tile(F;\Z^2)$ as an ergodic dynamical system (a tiling $A$ is thus viewed as an ergodic point process, rather than an individual set).  Using the dilation lemma \cite[Proposition 3.1]{B}, he analyses the spectral measure of the stationary point process $A$ to establish a structural decomposition 
     \begin{equation}\label{eq:Bstructure}
    \one_A=_\as  \pi_{h_1}(\one_A) +\dots +\pi_{h_m}(\one_A) +g,
     \end{equation}
     for this ergodic point process (almost surely), where $\pi_h(\one_A)$ is obtained (a.s.) by averaging the translates of the stationary point process $A$ in the direction $h$ and the function $g$ is (a.s.) periodic \cite[Theorem 3.3]{B}.
\end{remark}

We noted above that the dilation lemma (Lemma \ref{lem:dilation}) implies the existence of some kind of long-term correlations in a tiling. Theorem \ref{thm:structure} makes these correlations precise: every tiling can be described as a finite sum of singly-periodic functions with some additional properties. 
Thus, this theorem allows us to convert the periodic tiling conjecture into an equivalent conjecture involving a bounded number of singly-periodic functions $\varphi_{f,v}$, $f\in S_F, v\in V$ on $G$, which we found to be a more tractable formulation.
 The structural decomposition \eqref{eq:decomp} is particularly powerful in low dimensions. Indeed, single-periodicity is closer to periodicity when the dimension is low. However, one can suspect that as the dimension increases, this structural information becomes weaker, indicating possible wild behavior of high-dimensional tilings. 

 In the following sections, we describe how this structure theorem is exploited to prove not only positive results toward Conjecture \ref{con:ptc}, but also its failure.

\section{Structured monotilings: periodicity in low dimensions}\label{sec:low}

In this section we exploit Theorem \ref{thm:structure} to study the periodicity of tilings in low dimensions.

\subsection{The structure of one-dimensional monotilings.}
In \cite{N}, the periodicity of any tiling $A\in \Tile(F;\Z)$ was established  using a pigeonhole principle argument--resulting in an exponential upper bound $2^{\diam(F)}$ on the periods. 
Observe that using Remark \ref{rem:G0trivial}, the decomposition \eqref{eq:decomp} immediately reestablishes this result. Moreover, in \cite[Corollary 3.5]{GT} we used this decomposition to establish the existence of a \emph{universal} period of all tilings $A\in \Tile(F;\Z)$, which is bounded above by $|F|\diam(F)^{|F|-1}$. 

For a virtually-$\Z$ space $G=\Z\times G_0$ (where $G_0$ is a finite Abelian group), using the decomposition \eqref{eq:decomp} of the slices and the slices reparation machinery introduced in \cite[Section 5]{GT}, a tiling $A\in \Tile(F; G)$ can be ``repaired'' to be periodic by applying the machinery in \cite[Section 5]{GT}, reestablishing the periodic tiling conjecture in any virtually-$\Z$ space originally proved in \cite[Section 2]{GT2}.

The rest of this section is devoted to the two-dimensional case, which is significantly more complicated.

\subsection{The structure of two-dimensional monotilings.}

In \cite{B}, Bhattacharya proved the periodic tiling conjecture in $\Z^2$ using the decomposition \eqref{eq:Bstructure} by a rather complicated analysis. 
In this section, we describe how the conjecture can be easily deduced from the structural decomposition \eqref{eq:decomp}.

Let $F$ be a tile of $\Z^2$. Recall that to prove the periodic tiling conjecture in $\Z^2$ one has to locate a periodic tiling in the solution space $\Tile(F;\Z^2)$.
However, thanks to the following lemma \cite[Proposition 2.3]{B}, it suffices to establish the existence of a weakly-periodic tiling in $\Tile(F;\Z^2)$:

\begin{lemma}[From weak-periodicity to periodicity]\label{lem:weaktoptc}
Let $F$ be a finite subset of $\Z^2$ and suppose that $\Tile(F;\Z^2)$ contains a weakly-periodic set, then it also contains a periodic set.
\end{lemma}
 
 Suppose that $A\in \Tile(F;\Z^2)$ is weakly-periodic, i.e., $A=A_1\sqcup \dots\sqcup A_m$ for singly-periodic $A_1,\dots,A_m\subset \Z^2$.
     The proof of Lemma \ref{lem:weaktoptc} consists of two ingredients:
     \begin{enumerate}
         \item Sowing that for every $1\leq j\leq m$ the set $F\oplus A_j=E_j$ is periodic.
         \item Showing that if a set $S\subset \Z^2$ is singly-periodic and such that $F\oplus S$ is periodic, then there exists a periodic set $S'\subset \Z^2$ with $F\oplus S'=F\oplus S$. 
     \end{enumerate} 
      The first is proven in \cite[Proposition 2.1]{B} by algebraically analyzing the equation $\sum_{j=1}^m \one_{E_j}=1$. (Note that this part also follows from Theorem \ref{thm:structure}(ii).)
      The second part is proven in \cite[Section 2]{B} by adapting Newman's pigeonhole principle argument \cite{N}. In \cite[Section 5]{GT} an alternative proof of the second part was suggested, analyzing the one-dimensional slices and repairing them to make them compatible while preserving the tiling; this approach allows better control of the size of the periods obtained.

Our aim is therefore to locate a weakly-periodic tiling in $\Tile(F;\Z^2)$.
 Recall that for every tiling $A\in \Tile(F;\Z^2)$ Theorem \ref{thm:structure} gives a decomposition
 \begin{equation}\label{eq:Adecomp}
     \one_A=1-\sum_{v\in V} \varphi_v,
 \end{equation}
 where $V$ is a finite subset of primitive, linearly independent vectors $\Z^2\times \{0\}$, and for each $v\in V$, the function $\varphi_{v}: \Z^2 \to [0,1]$ is a $\langle qv\rangle$-periodic for sufficiently divisible $q$. This decomposition establishes a ``functional-weak-periodicity'' of $A$, presenting it as a finite sum of singly-periodic functions. In \cite[Theorem 3.4]{GT,GT25}, we showed that for every tiling $A\in \Tile(F;\Z^2)$ we can replace the functions $\varphi_v$, $v\in V$ in the decomposition \eqref{eq:Adecomp} by \emph{indicator functions} (supported on disjoint periodic sets), yielding a stronger structure theorem:

\begin{theorem}[Structure theorem in $\Z^2$]\label{thm:2structure}
    Let $F$ be a finite subset of $\Z^2$. There exist a finite set $V$ of primitive, mutually incommensurable vectors $\Z^2$, an integer $N\geq 1$ and a sufficiently divisible number $q$, such that for every $A\in \Tile(F;\Z^2)$ there is a decomposition 
    \begin{equation}\label{eq:2decomp}
        A = \bigsqcup_{v \in V} A_v,
    \end{equation} 
    where for each $v\in V$ the set $A_v$ is a $\langle qv\rangle$-periodic and  the set $(F\cap (x+\langle v\rangle)) \oplus A_v$ is $N\Z^2$-periodic for every $x\in \Z^2$.
\end{theorem}

This theorem follows from \cite[Theorem 1.4]{GT}. Given a finite $F\subset \Z^2$, let $V$ be as in Theorem \ref{thm:structure} and suppose $A\in \Tile(F;\Z^2)$. The proof in \cite[Section 4]{GT} consists of the following main steps: 
\begin{itemize}
\item \emph{Step 1.}  Showing that there exist  $N'\in \N$ and a vector $e\in N'\Z^2$ that is incommensurable with each of the vectors $v\in V$, such that for every $x\in \Z^2$ and $v\in V$ the one-dimensional function 
\begin{equation}\label{eq:poly}
    n \in \Z \mapsto \varphi_v (n e+x) \Mod 1
\end{equation} 
is a polynomial. The argument relies on iterated discrete differentiation of \eqref{eq:Adecomp} in the directions of (all but one of) the periods, which gives that the derivative of order $(|V|-1)$ of $\varphi_v$ along a suitable vector $e$, $\partial_e^{|V|-1} \varphi_v$ vanishes modulo $1$ for every $v\in V$-- characterizing polynomial behavior.

\item \emph{Step 2.} By classical Weyl equidistribution theory, each of the functions in \eqref{eq:poly} must be either periodic or equidistributed. Thus, this dichotomy exhausts all possible structural behaviors of the functions $\varphi_v$, $v\in V$, when restricted to each coset of $N'\Z^2$.

\item \emph{Step 3.} By \eqref{eq:Adecomp} equidistribution scenarios cannot occur for level-one tilings. Roughly speaking, there is simply not enough ``room'' for more than one of the functions $\varphi_v$, $v \in V$ to equidistribute (see \cite[Proposition 4.1(iii)]{GT}). On the other hand, since the sum of these functions is ${0,1}$-valued, the existence of a single equidistributed function would force it to be canceled by other equidistributed functions, leading to a contradiction.
Consequently, all of these polynomials must, in fact, be periodic.

\item \emph{Step 4.} We conclude that there exists some $N''$ divisible by $N'$ such that, for each coset $\Gamma$ of the lattice $N''\Z^2$, the restriction $\one_{A\cap \Gamma}=\one_A\one_\Gamma$ of $\one_A$ to $\Gamma$ falls into one of two cases: either it is periodic, or it takes the form $1-\varphi_v$ for some $v\in V$. In the latter case, it is $\langle qv\rangle$-periodic for sufficiently divisible $q$, which yields the desired decomposition \eqref{eq:2decomp}.

\item \emph{Step 5.} Finally, a further application of the dilation lemma \ref{lem:dilation} to the tiling $F\oplus A_v=E_v$ (where the set $E_v$ is periodic by Theorem \ref{thm:structure}(ii)) gives the $N\Z^2$-periodicity of the set $(F\cap (x+\langle v\rangle))\oplus A_v$ for every $x\in \Z^2$ and some $N$ divisible by $N''$. See \cite[Lemma 5.1]{GT} or \cite[Theorem 3.4(ii)]{GT25} for details.
\end{itemize}

Theorem \ref{thm:2structure} implies that \emph{every} tiling $A\in \Tile(F;\Z^2)$ is weakly-periodic \cite[Theorem 1.3(i)]{GT}. In turn, by Lemma \ref{lem:weaktoptc}, we conclude that Conjecture \ref{con:ptc} holds in $\Z^2$, with a universal bound on the period \cite[Theorem 1.5]{GT}. As a result, we can conclude the decidability of translational tilings in $\Z^2$, with a bound on the computational complexity of the problem \cite[Corollary 1.6]{GT}.

\begin{remark}\label{rem:highdim}
  The argument in the proof of Theorem \ref{thm:2structure} does not extend to $\Z^d$ for $d\geq3$ since, generally, the functions we get by applying the discrete derivative machinery (Step 1 above) are $d-1$ dimensional. In the two-dimensional case, they are therefore one-dimensional, as described in \eqref{eq:poly}, and hence we can deduce that they are polynomials from the fact that their discrete derivative (to some order) vanishes. When $d\geq 3$, these functions are of dimension greater than one, and so although their (partial) discrete derivative in $d-1$ directions (to some order) vanishes, this no longer implies that they are polynomials. For instance, for any $h,r\colon \Z\to \R$ the map
  $$(n,m)\in \Z^2\mapsto h(n)+r(m)$$
  vanishes under $\partial_{(0,1)}\partial_{(1,0)}$, with no restrictions on the functions $h,r$; in particular, they need not be polynomials.
  In turn, the Weyl dichotomy of periodicity versus equidistribution may not cover all possible structural behaviors of these functions.
Thus, new methods are required to understand the full range of possible structures of the solutions in three dimensions and beyond. In particular, to date, the following question remains open.  \begin{question}\label{q:Z3} Does Conjecture \ref{con:ptc} hold in $\Z^3$? \end{question}
\end{remark}

\subsection{Periodicity of soft tilings in $\Z^2$.}

In \cite{KL}, a natural generalization of the study of the structure of tilings to, so-called, \emph{soft tilings} was suggested, where the tile set is replaced with a compactly supported function and the set being tiled is replaced with a periodic function. 

\begin{definition}[Discrete soft tiling in $\Z^2$]
Let $f\colon \Z^2\to \Z$ be compactly supported, and let $g\colon \Z^2\to \Z$ be periodic. A soft tiling of $g$ by $f$ is a set $A\subset \Z^2$ such that $f *\one_A=g$.
\end{definition}

A special case of soft tilings is \emph{higher level tilings}, also known as \emph{multi tilings}. In this case, the function being tiled is a constant $k>1$. In particular, a finite set $F \subset \Z^2$ \emph{tiles $\Z^2$ at level $k$} along a tiling set $A$ if $\one_F *\one_A =k$, i.e., every point in $\Z^2$ belongs to exactly $k$ elements of the family $\{F +a\}_{a\in A}$. (See. e.g., \cite{Bolle,Gravin0,Gravin,Kolountzakis2000,Liu2021,Robins}, and the references therein for various aspects of the study of multi tilings.)

\begin{example}[Multi tiling]
     Let $F=\{0,2,3\}\times \{0,1\}$. Then clearly $\Tile(F;\Z^2)$ is empty. However, for $A=\Z\times 2\Z$ we have $\one_F * \one_A=3$, which means that $F$ tiles $\Z^2$ at level three along $A$. See \cite{Robins} for more multi tiling examples. Also, note that trivially every finite set $F\subset \Z^2$ multi tiles $\Z^2$ at level $|F|$ by translations along $A=\Z^2$.
\end{example}

In \cite[Theorem 1.3(ii)]{GT} we showed that Theorem \ref{thm:2structure} fails for multi tilings in $\Z^2$. Indeed, in this case, the higher level leaves enough ``room'' for equidistribution scenarios to occur in the tiling, resulting in a non-structured behavior. In particular, multi tilings in $\Z^2$ do not need to be weakly-periodic, and thus we cannot use the machinery of the proof of Lemma \ref{lem:weaktoptc} to conclude the existence of a periodic solution.

\begin{example}[Equidistributed multi tiling]\label{ex:multi}
Let $$F=\left\{\epsilon_1(0,2)+\epsilon_2(1,0)+\epsilon_3(2,-2)\colon \epsilon_1,\epsilon_2,\epsilon_3\in \{0,1\}\right\}$$ (see Figure \ref{fig:multi}),
and for every $\alpha\in \R$ let 
$$A_\alpha \coloneqq \left\{(n,m)\colon (-1)^{\lfloor m/2\rfloor +n} \left(\{\alpha n\}+\{\alpha m\}-\{\alpha(n+m)\}-1/2\right) > 0\right\}.$$
 Then we have $\one_F *\one_{A_\alpha}=4$, which means that $F$ tiles $\Z^2$ at level four along $A_\alpha$. Observe that when $\alpha$ is irrational, the set $A_\alpha$ is not weakly-periodic (see \cite[Section 2]{GT} for details).\footnote{See also \cite[Example 3.4.6]{sz} for a related \emph{Bohr set}-type construction in the context of Nivat's conjecture on the structure of low complexity configurations in $\Z^2$.} Nevertheless, in this case $F$ does admit a \emph{periodic} level four tiling as for every rational $\alpha$ the tiling set $A_\alpha$ is periodic. 
\end{example} 

\begin{figure}[ht]
    \centering
    \includegraphics[width=2.5cm]{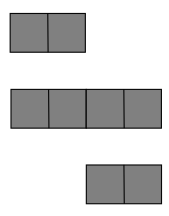}
    \caption{An illustration of the set $F$ in Example \ref{ex:multi} (where a point $x$ in $\Z^2$ is represented by $x+[0,1]^2$).}
    \label{fig:multi}
\end{figure}

In our recent work \cite[Theorem 1.9]{GT25} we finally establish the periodic tiling conjecture for soft tilings in $\Z^2$. As a result, we also deduce the decidability of soft (and multi) tilings in $\Z^2$ \cite[Corollary 1.10]{GT25}.

The proof of \cite[Theorem 1.9]{GT25} is quite subtle; it consists of a fine analysis of the structure of soft tilings, using a suitable extension of Theorem \ref{thm:structure} to integer tilings.
We first apply ergodic and probabilistic methods (ergodic limits and first and second moments) to clean up any given solution $\one_A$, transforming it into a normal form \cite[Proposition 8.7]{GT25} that is more convenient to analyze, as it provides a complete separation between the structured and unstructured components of the tiling. Then, using equidistribution theory together with a sequence of reductions to the description of the solution space (up to null sets) using our structure theory for soft tilings, we managed to describe a normal solution as a finite system of first-order sentences in the language of linear inequalities over the rationals (in the reals). Since this system admits a real solution, it must also admit a rational solution.
By rationalizing the possible equidistribution scenarios, we gain periodicity.

\section{Discovery of wild translational monotilings}\label{sec:high}
In the previous section, we saw that in $\Z$, virtually-$\Z$ and also in $\Z^2$, Theorem \ref{thm:structure} is strong enough to imply the periodic tiling conjecture (Conjecture \ref{con:ptc}). But is it strong enough to infer the periodic tiling conjecture in higher dimension? In this section, we discover not only that the answer is no, but also that, in fact, the periodic tiling conjecture is false even when restricted to virtually-$\Z^2$ spaces. We describe the behind-the-scenes process (which has not been previously published) that ultimately led to the construction of the counterexample in \cite{GT22}.

\subsection{Reversible structure theorem in virtually-$\Z^2$ spaces.}
In Section \ref{sec:low} we saw that in $\Z^2$, the functions on the right hand side of \eqref{eq:decomp} can be replaced with indicator functions (Theorem \ref{thm:2structure}). As noted in Remark \ref{rem:highdim}, the proof of this theorem does not extend to $\Z^d$ for $d\geq 3$. However, in virtually-$\Z^2$ spaces, applying the machinery in the proof of Theorem \ref{thm:2structure} to the decomposition \eqref{eq:decomp} of \emph{slices} $\one_{F_f}*\one_A$ of the tiling (rather than the tiling $\one_A$ itself), we obtain a reversible structure theorem, where the right hand side of \eqref{eq:decomp} is replaced by indicator functions. 

\begin{theorem}[Reversible structure theorem in $G=\Z^2\times G_0$]\label{thm:vir2structure}
    Let $G_0$ be a finite Abelian group and $F$ be a finite subset of $G=\Z^2\times G_0$.  Then there exist a sufficiently divisible $q$, a lattice $N\Z^2$ and  a finite set  $V$ of primitive, nonzero, mutually incommensurable vectors in $\Z^2$ such that for every $A\subset G$ the following are equivalent:
\begin{enumerate}
    \item[(i)]  $F \oplus A = G$.  
    \item[(ii)] For every $f\in S_F$ there is a decomposition 
    \begin{equation}\label{eq:vir2decomp}
        \one_{F_f} * \one_A = \sum_{v \in V} \one_{E_{f,v}},
    \end{equation}  where for each $v\in V$, the set $E_{f,v}$ is a $\langle qv\rangle$-periodic (viewing $\Z^2$ as a subgroup of $G$). Moreover,  the sets $E_{f,v}$ are disjoint and for every $v\in V$ the set $E_v\coloneqq \bigsqcup_{f\in S_F} E_{f,v}$ is $N\Z^2$-periodic and $\bigsqcup_{v\in V} E_v = G$.
\end{enumerate}  
\end{theorem}

We then used this reversible statement to reduce Conjecture \ref{con:ptc} in virtually-$\Z^2$ spaces to a new conjecture on finite colorings of the integers, which only involves understanding the properties of one-dimensional sets, rather than two-dimensional sets, and is therefore more tractable, as described below.

\subsection{Periodic coloring conjecture.}
Given an integer $N \geq 1$, a finite set $V$ of primitive, nonzero, mutually incommensurable vectors in $\Z^2$, 
a finite (color) set $\Sigma$ and a subset $\Omega$ of $ \Sigma^{|V|}\times (\Z/N\Z)^2$, we set  $\overline{V}\coloneqq\{\bar{v}\colon v\in V\}$, where $\bar{v}$ is a fixed primitive vector orthogonal to $v$ and define the set $\Solution(V, N, \Sigma, \Omega )$
to be the collection of all tuples of functions $C_{v}: \Z \to \Sigma$, $v \in V$ such that there is a two-dimensional \emph{clock} $\sigma\colon (\Z/N\Z)^2\to (\Z/N\Z)^2$ (in each coordinate, $\sigma$ is increasing by one every time the input advances by one \cite[Example 4.8]{GT22}), for which
	\begin{equation}\label{inomega}
	    ((C_{v}(x \cdot \bar{v}))_{v \in V},\sigma(x \Mod{N\Z^2}))  \in \Omega
	\end{equation}
	for every $x \in \Z^2$.  We say that this solution set $\Solution(V, N, \Sigma, \Omega )$ is \emph{aperiodic} if it is nonempty, but does \emph{not} contain any tuples $(C_{v})_{v \in V}$ with \emph{all} of the functions $C_{v}\colon \Z\to \Sigma$ being periodic.

	\begin{example}  Let $M$ be a natural number.  Consider the set of all triples $(r,g,h)$ of functions $r,g,h \colon \Z \to \Z/M\Z$ obeying the relation
	$$ h(n+m) \in r(n) + g(m) +  (\{0,1\}+M\Z)$$
	for all $n,m \in \Z$. This set is of the form $ \Solution( V, N, \Sigma, \Omega )$, where $\overline{V} = \{(1,0),(0,1),(1,1)\}$, $N=1$, $\Sigma = \Z/M\Z$, and 
	$$ \Omega := \{ (a,b,c) \in (\Z/M\Z)^3\colon  c \in a+b+\{0\Mod M,1 \Mod M\} \}\times \{0\}.$$
	For any real number $\alpha$, one can create an element $(r,g,h)$ of this set by taking $r(j) = g(j) = h(j) = \lfloor \alpha j \rfloor \Mod M$ for all $j \in \Z$.  If $\alpha$ is rational, this is periodic.  In particular, in this case $ \Solution( V, N, \Sigma, \Omega )$ is not aperiodic.
	\end{example}
	
	One can show that the periodic tiling conjecture \ref{con:ptc} holds in virtually-$\Z^2$ spaces if and only if the following alternative periodic coloring conjecture holds:
	
	\begin{conjecture}[Periodic coloring conjecture]\label{con:altptc}  
    For any $V, N, \Sigma, \Omega$ as above, the solution set $\Solution(V, N, \Sigma, \Omega)$ is not aperiodic.
	\end{conjecture}

Using Theorem \ref{thm:vir2structure} together with a slicing analysis (as in \cite[Section 5]{GT} and \cite[Sections 3, 9]{GT25}), we can describe any tiling instance $\Tile(F;\Z^2\times G_0)$ of Conjecture \ref{con:ptc} as a coloring instance $\Solution(V,N,\Sigma,\Omega)$ of Conjecture \ref{con:altptc}. 
Roughly speaking, $V$ is the same set as in Theorem \ref{thm:vir2structure}, $\Sigma$ is a suitable subset of $\{0,1\}^{G_0}$ and the coloring functions $C_v$, $v\in V$ are determined by the projections $\pi_{qv}$, $v\in V$ of a tiling $A\in \Tile(F;\Z^2\times G_0)$. For each $x\in \Z^2$, the behavior of $A$ in the slice $\{x\}\times G_0$ corresponds to the point $(C_{v}(x \cdot \bar{v}))_{v \in V}$ of $(C_v)_{v\in V} \in \Solution(V,N,\Sigma,\Omega)$ (and $\Omega$ is defined to reflect this correspondence). 
In particular, $\Tile(F;\Z^2\times G_0)$ is aperiodic if and only if $\Solution(V,N,\Sigma,\Omega)$ is aperiodic.

Conversely, the tiling language of \cite{GT2} (see Section~\ref{subsec:language}) allows us to express any coloring instance $\Solution(V,N,\Sigma,\Omega)$ of Conjecture \ref{con:ptc} as a monotiling equation \begin{equation}\label{eq:eq}
    F\oplus X=\Z^2\times G_0.
\end{equation} 
The key here is that the definition of an element of $\Solution (V,N,\Sigma,\Omega)$ consists of the sentence \eqref{inomega}, which, as mentioned in Remark \ref{rem:obstacle}, is expressible as a monotiling equation. 
This encoding preserves aperiodicity, demonstrating that the set $\Solution(V,N,\Sigma,\Omega)$ is aperiodic if and only if the corresponding tiling equation \eqref{eq:eq} is aperiodic  \cite{GT22}.

The remainder of the section is devoted to resolving Conjecture \ref{con:altptc}, elucidating the process through a discussion of failed attempts.

\subsection{Failed van der Waerden-type attempt.}

One can attempt to establish the periodic coloring conjecture \ref{con:altptc} (and, in turn, the periodic tiling conjecture \ref{con:ptc} in virtually-$\Z^2$ spaces)  using the following van der Waerden-type approach.

For simplicity, assume $\overline{V} = \{(1,0), (0,1), (1,1)\}$ (the argument extends to other choices of $V$). Then, an element in $\Solution( V, N, \Sigma, \Omega )$ consists of a triple of functions $r,g,h \colon \Z \to \Sigma$ obeying the constraint
\begin{equation}\label{rgh}
(r(n), g(m), h(n+m), \sigma((n,m) \Mod{N\Z^2})) \in \Omega
\end{equation}
for all $(n,m) \in \Z^2$, for some clock function $\sigma$. We then wish to locate periodic functions $\tilde r, \tilde g, \tilde h \colon \Z \to \Sigma$ obeying \eqref{rgh}. Let $(r,g,h)\in \Solution(V,N,\Sigma,\Omega)$.  Then, by the van der Waerden theorem \cite{vanderWaerden1927}, there exists an arithmetic progression $a, a+d, \dots, a+(2N-2)d$ of length $(2N-1)$ in $\Z$ such that
\begin{equation}\label{period}
 h(a+jd) = h(a+(N+j)d)
 \end{equation}
for all $j=0,\dots,N-2$.
Suppose we could find such a progression with the additional requirement $d=1\Mod N$. We could then define $\N\Z$-periodic functions $\tilde r, \tilde g, \tilde h$ as
$$ \tilde r( Ns + j ) \coloneqq r(a+jd) \quad \text{and} \quad \tilde g( Ns + j ) \coloneqq g(jd)$$
for every $s \in \Z$ and $j=0,\dots,N-1$, and 
$$ \tilde h( Ns + j ) \coloneqq h(a+jd)$$
for all $s \in \Z$ and $j=0,\dots,2N-2$ (which is well-defined thanks to \eqref{period}).
Then, by \eqref{rgh} and the construction of $(\tilde r,\tilde g,\tilde h)$, we would have
$$ (\tilde r(n), \tilde g(m), \tilde h(n+m), \sigma((n,m)\Mod {N\Z^2})) \in \Omega$$
for every $(n,m)\in \Z^2$ such that
\begin{equation}\label{nm}
    n = Ns_n + j_n,\quad m = Ns_m + j_m
\end{equation}
with $s_n,s_m\in \Z$ and $j_n,j_m =0,\dots,N-1$. 
But since all integers $n,m$ can be represented as \eqref{nm}, we would obtain that the functions $\tilde r, \tilde g, \tilde h$ obey \eqref{rgh} for all $(n,m)\in \Z^2$ and hence $\Solution(V,N,\Sigma, \Omega)$ contains the periodic solution $(\tilde r,\tilde g,\tilde h)$.

Therefore, if we could prove a strengthening of the van der Waerden theorem stating that for any finite coloring of the integers $h: \Z \to \Sigma$ and every $M,N\in \N$ there exists a length $MN$ arithmetic progression $a, a+d, \dots, a+MNd$ such that $d = 1 \Mod N$ and $h(a+jd) = h(a+(N+j)d)$ for all $0 \leq j \leq (M-1)N$, then we would establish Conjecture \ref{con:altptc} and, in turn, Conjecture \ref{con:ptc} in virtually-$\Z^2$ spaces.

However, it turns out that this strong van der Waerden-type statement fails to hold. For instance, let $p$ be an odd prime, $\Sigma_p\coloneqq (\Z/p\Z)^\times$ and $f_p\colon \Z\to \Sigma_p$ be the function 
\begin{equation}\label{eq:fp}
    f_p(n)\coloneqq\frac{n}{p^{\nu_p(n)}} \Mod p
\end{equation}
for every $n\in \Z\setminus \{0\}$, where $\nu_p(n)$ is the number of times $p$ divides $n$, and we set $f_p(0)\coloneqq 1\Mod p$. Then $f_p$ is a counterexample to this van der Waerden-type statement, since for every $N\in \N$,  $a\in \Z$, and integer $d$ with  $d = 1 \Mod N$, there exists $0 \leq j \leq p^2 N$ such that 
$$f_p(a+jd) \neq f_p(a+(N+j)d).$$

The failure of this combinatorial attempt marked a turning point: it led us to consider the function $f_p$ as a possible counterexample to Conjecture \ref{con:altptc}.
Transitioning to a dynamical point of view, we analyzed the \emph{orbit closure} of $f_p$, which became the key ingredient in our construction of aperiodic Sudoku puzzles--later translated (using our tiling language) into a counterexample to the periodic tiling conjecture \ref{con:ptc}.

\subsection{From aperiodic subshift to aperiodic Sudoku to aperiodic monotiling.}

Let $p$ be an odd prime. The function $f_p$ in \eqref{eq:fp} is an example of a \emph{limit-periodic function} \cite{godreche1989sphinx} (in particular, it is an \emph{almost periodic} function in the sense of Besicovitch \cite{besicovitch1926gapf}): 
Other than in the residue class $0 \Mod {p}$, $f_p$ behaves as a clock: it is an affine map from $\Z\setminus p\Z$ to $\Sigma_p$ that increases by one every time the input advances by one; and inside this residue class ($0 \Mod {p}$) but outside its subresidue class $0\Mod {p^2}$, $f_p$ is again an affine map from $p\Z\setminus p^2\Z$ to $\Sigma_p$ that increases by $1$ every time the input advances by $1$; and so on (see \cite[Figure 4.1]{GTan}). Thus, the subshift $f_p$ generates is a \emph{Toeplitz subshift}. 

Let $\cS_p\subset (\Sigma_p)^\Z$ denote the class of functions $g: \Z \to \Sigma_p$ that are either constant, or have a similar $p$-adic affine structure. Namely, $g\in\cS_p$ whenever there is a step $h\in \Sigma_p$ such that outside of one coset $\Gamma_1$ of $p\Z$, $g$ is affine, with $g$ increasing by $h$ when the input advances by $1$; in $\Gamma_1$ but outside of one coset $\Gamma_2\subset \Gamma_1$ of $p^2\Z$, $g$ is again affine, and advancing by $h$ when the input advances by $p$; and so on. This set is the \emph{orbit closure} of $f_p$, in the sense that the functions $n\mapsto f_p(an+b)$ for $(a,b)\in \Z^2$ form a dense subset of $\cS_p$ (which is translation and dilation invariant). Thus, it contains the subshift generated by $f_p$ (see \cite[Proposition 13]{jv}), which is aperiodic in the sense that it is not empty and none of its points is periodic.

More generally, the only periodic elements of $\cS_p$ are the constant ones. Hence, to disprove Conjecture \ref{con:altptc}--and consequently Conjecture \ref{con:ptc}--it suffices to exhibit a system of constraints of the type described in Conjecture \ref{con:altptc} that determines a nonempty class of non-constant solutions in $\cS_p$ (for instance, the points in the subshift generated by $f_p$). We constructed precisely such a system, which we view as a kind of \emph{$p$-adic Sudoku puzzle}: the rules are local, yet they enforce a rich global $p$-adic affine structure, which is non-periodic.

Fix $p$ to be a sufficiently large prime and let $M\coloneqq p^2$.  We define $V$ as the set 
$$ V \coloneqq  \{ (1,-n): 1 \leq n \leq M\}$$
of primitive, nonzero, mutually incommensurable vectors in $\Z^2$, and consider the set of tuples $(C_v)_{v \in V}$ of functions $C_v: \Z \to \Sigma_p$ obeying the following conditions:
\begin{itemize}
    \item[(i)] For every $n=1,\dots,M$, the function $C_{(1,-n)}$ is not constant. 
    \item[(ii)] For every $(a,b)\in \Z^2$ the map 
$$n\mapsto C_{(1,-n)}((a,b)\cdot (n,1))=C_{(1,-n)}(an+b)$$
 from $\{1,\dots,M\}$ to $\Sigma_p$ agrees with a $\{1,\dots,M\}$-cutoff of an element of $\cS_p$.
\end{itemize}

Note that this set is of the form $\Solution(V,p,\Sigma_p,\Omega_p)$ for a finite set $\Omega_p\subset (\Sigma_p)^M\times (\Z/p\Z)^2$ defined based on conditions (i) and (ii).
It can be thought of as the following \emph{$p$-adic Sudoku puzzle}: We fix the \emph{Sudoku board} $\mathbb{B}=\{1,\dots,M\}\times \Z$. A function $S\colon \mathbb{B}\to \Sigma_p$ is a \emph{Sudoku solution} if none of its \emph{columns},
\begin{equation}\label{eq:column}
    m \mapsto S(n,m), \quad (n=1,\dots,M)
\end{equation}
is constant (e.g., a clock or periodized permutation on $\Z\setminus\Gamma_n$ for some coset $\Gamma_n$ of $p\Z$), and for every $a,b\in \Z$, the restriction of $S$ to a \emph{line} of slope $a$ and intercept $b$ on the Sudoku board:
$$S_{(a,b)}(n)\coloneqq S(n,na+b), \quad n=1,\dots,M$$
agrees with a $\{1,\dots,M\}$-cutoff of an element of $\cS_p$. So, the $n^{\text{th}}$ column \eqref{eq:column} of a Sudoku solution $S$ corresponds to the function $C_{(1,-n)}$ in our previous notation. 

This Sudoku puzzle admits solutions; for instance, the function $S\colon \mathbb{B}\to \Sigma_p$ defined as $S(n,m)\coloneqq f_p(m)$ for every $n=1,\dots,M$ solves this Sudoku puzzle (we refer to this Sudoku solution as a \emph{standard} solution). See Figure \ref{fig:p5}.

\begin{figure}[ht]
    \centering
    \includegraphics[width=11cm]{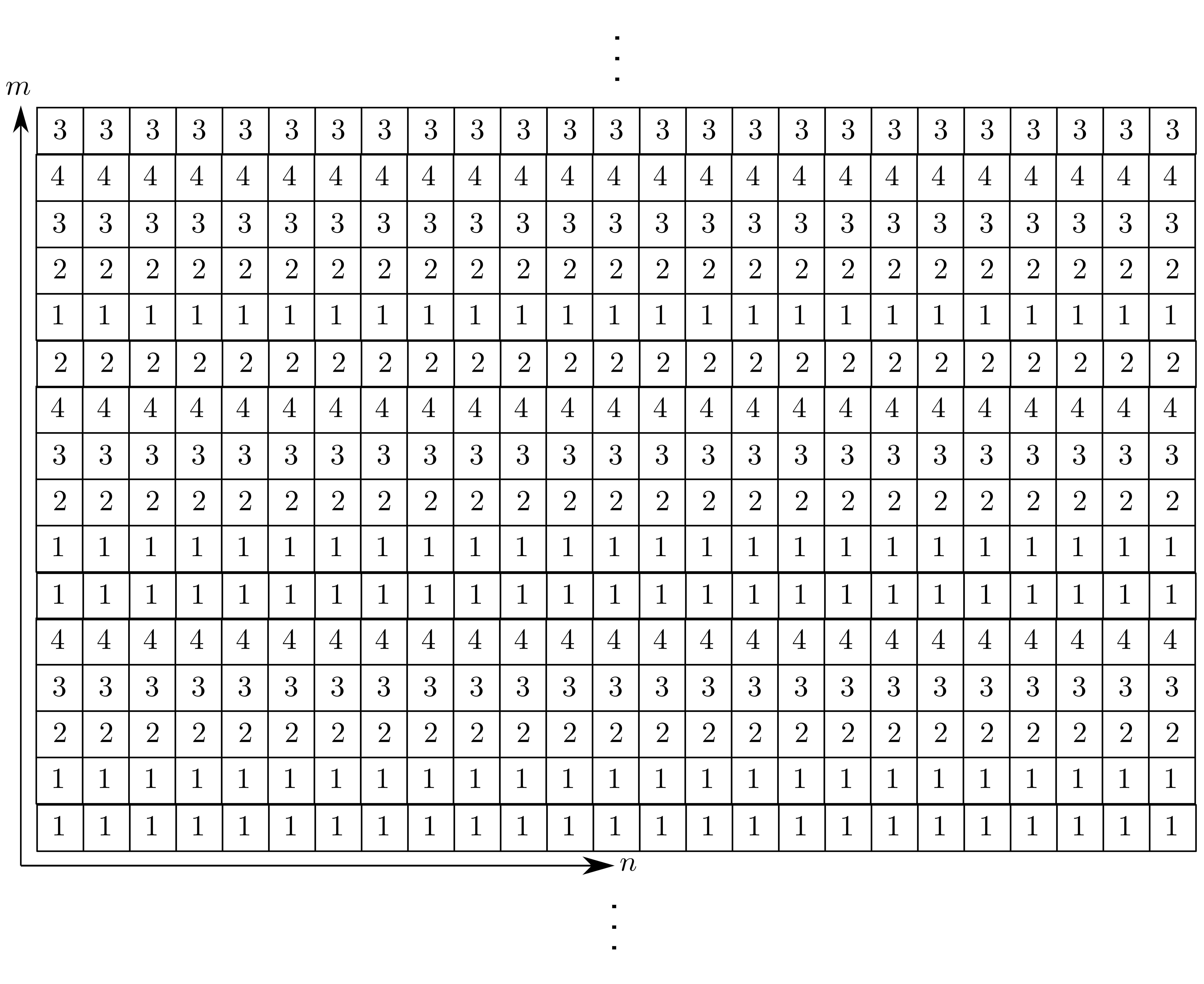}
    \caption{Standard $5$-adic Sudoku solution.}
    \label{fig:p5}
\end{figure} 

We then show that the $p$-adic nature of the Sudoku constraints impose a hierarchical structure, which forces the columns--lines of infinite slope of \emph{any} Sudoku solution $S$ to be non-constant elements of $\cS_p$ (in fact, each of these columns is a point in the Toeplitz subshift generated by $f_p$). We see that the $p$-adic Sudoku puzzle has solutions, but none of them can be periodic. This means that this puzzle is \emph{aperiodic}, providing a counterexample to the periodic coloring conjecture \ref{con:altptc}.

 Using our tiling language, we translate this puzzle into a monotiling equation $\Tile(F;\Z^2\times E)$ in a virtually-$\Z^2$ space\footnote{Here $E$ is a subset of a finite Abelian group $G_0$. In order to obtain a tiling of the entire group $\Z^2 \times G_0$, in our actual construction in \cite{GT22} we replace the large prime $p$ with a large power of $2$. This makes the analysis of Sudoku solutions technically more involved. For clarity of exposition, we therefore focus here on the odd prime case, which captures the essence of the construction.}. 
 Every solution $S\colon \mathbb{B}\to \Sigma$ to the $p$-adic Sudoku puzzle is converted into a tiling $A\in \Tile(F;\Z^2\times E)$ and vice versa. Informally, the $(1,-n)$-invariant component of the tiling corresponds to the $n^{\text{th}}$ column of the Sudoku solution $S$ and for every $(a,b)\in \Z^2$, the behavior of the tiling equation $F\oplus A$ in the slice $\{(a,b)\}\times E$ corresponds to the line of slope $a$ and intercept $b$:
 $$S_{(a,b)}=(S(1,a+b),S(2,2a+b),\dots,S(M,Ma+b)),$$ in the Sudoku solution $S$ (resembling the classical projective duality between points and lines in the plane  \cite[Remark 7.10]{GT22}).
 By the nature of this encoding, this monotiling equation is also aperiodic and therefore yields a counterexample to the periodic tiling conjecture \ref{con:ptc} (see \cite{GT22} for details).

\subsection{Undecidability of translational monotilings.}

A proof of the periodic tiling conjecture implies the decidability of tilings (see Section \ref{sec:domino}). The converse, however, is subtler: while the presence of an aperiodic tiling does not automatically lead to the undecidability of tilings, such an aperiodic construction is often the first step toward establishing an undecidability result. Motivated by this, once we had constructed the counterexample to Conjecture \ref{con:ptc}, we turned to the question of whether it could be used to prove the undecidability of translational monotilings.

In \cite{AanderaaLewis1974}, Aanderaa and Lewis employed a $p_1 \times p_2$-adic structure (with sufficiently large distinct primes $p_1,p_2$) to establish the undecidability of the \emph{empty distance subshift problem}. Exploiting the hierarchy of the $p_1\times p_2$-adic scales, they then gave an alternative proof of the undecidability of the Wang domino problem: each $p_1 \times p_2$-adic scale,  $(a,b)\in \N\times \N$, of the subshift was decorated to encode the Wang tile covering the corresponding point $(a,b)\in \N^2$. (See \cite[Section 4]{jv} for an exposition on this construction.)

Note that the structure of the solutions to our $p$-adic Sudoku puzzle \cite{GT22} is similar in nature to this $p_1\times p_2$-adic structure\footnote{Emmanuel Jeandel brought the construction of Aanderaa and Lewis to our attention and pointed out the similar structure.}.
Thus, inspired by \cite{AanderaaLewis1974,Lewis1979}, we introduce in \cite{GT23} the notion of \emph{decorated $p_1\times p_2$-adic Sudoku puzzle}. We then show that any instance of the domino problem can be encoded as such a decorated Sudoku  \cite[Theorem 3.2]{GT23}. See Figure \ref{fig:decorated}.

\begin{figure}[ht]
    \centering
    \includegraphics[width=10cm]{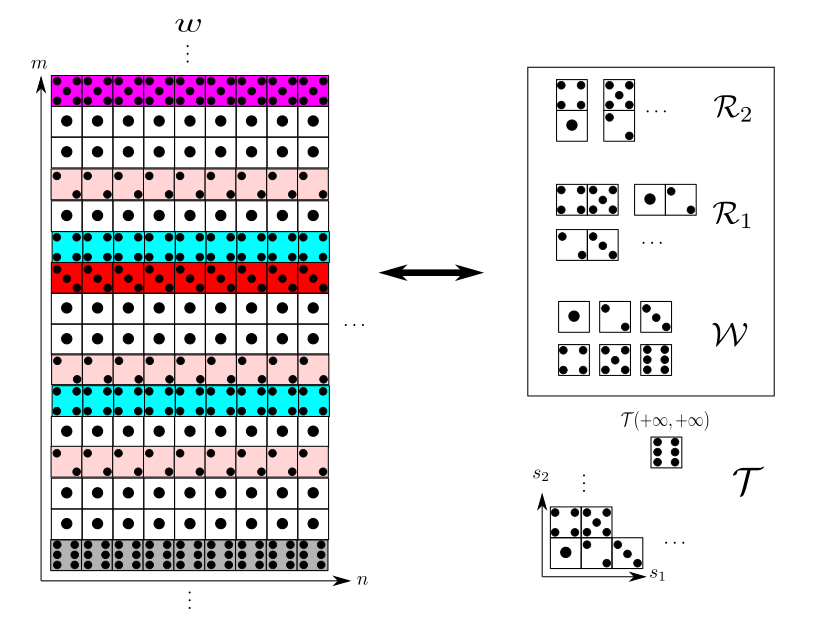}
    \caption{On the right, a given domino problem and a piece of a solution on $\{(0,0),(1,0),(2,0),(0,1),(1,1)\}$. On the left is the corresponding standard decorated $3\times 5$-adic Sudoku solution; the white cells represent the $(0,0)$ scale of the Sudoku solution, the pink cells represent the $(1,0)$ scale, red represents the scale $(2,0)$, blue represents the scale $(0,1)$ and purple represents the $(1,1)$ scale.}
    \label{fig:decorated}
\end{figure} 
By the undecidability of the domino problem, it follows that decorated $p_1 \times p_2$-adic Sudoku puzzles are also undecidable; that is, there is no algorithm that, given a decorated $p_1 \times p_2$-adic Sudoku puzzle, determines whether it admits a solution. 

Using our tiling language, we further encode any decorated $p_1 \times p_2$-adic Sudoku puzzle as a monotiling equation in $\Z^2 \times G_0$, for a suitable finite Abelian group $G_0$ depending on the given Sudoku puzzle. Consequently, since decorated $p_1 \times p_2$-adic Sudoku puzzles are undecidable, we obtain the undecidability of translational monotilings in virtually-$\Z^2$ spaces \cite[Theorem 1.3]{GT23}.  See Figure \ref{fig:logic}. 

\begin{figure}[ht]
    \centering
    \includegraphics[width=5.5cm]{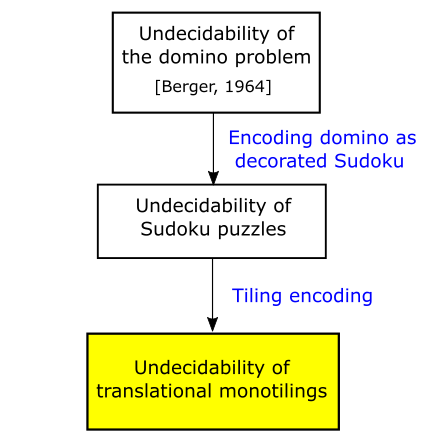}
    \caption{A high-level overview of the proof of the undecidability of translational monotilings.}
    \label{fig:logic}
\end{figure} 

\begin{remark}
    Note that in our undecidability result \cite[Theorem 1.3]{GT23} the group $\Z^2\times G_0$ is not fixed; the finite Abelian group $G_0$ is provided as part of the algorithm input. Indeed, the group $G_0$ is determined by the decoration rules of the Sudoku, which is determined by the original domino problem. We reduce the number of tiles in the domino problem to one at the cost of allowing the group $G_0$ to vary and to have an arbitrarily large rank.  
Using \cite[Corollary 1.3]{bgu}, we can pull back our undecidability result for translational monotilings in virtually-$\Z^2$ spaces to deduce the undecidability of translational monotilings in $\Z^d$, where $d$ is not fixed, but is provided as part of the input and can be arbitrarily large. Thus, the following question remains open.

\begin{question}
    Is there $d\geq 3$ such that translational monotilings are undecidable in $\Z^d$?
\end{question}
\end{remark}

\section{Continuous translational monotilings}

Let $\Omega\subset \R^d$ be a bounded measurable set of positive measures. We say that $\Omega$ \emph{tiles $\R^d$ by translations} if there exists a \emph{tiling} $T\subset \R^d$ such that almost every point in $\R^d$ is covered by \emph{exactly} one of the sets $\Omega +t$, $t\in T$; i.e., these sets \emph{partition} $\R^d$ up to null sets, or, equivalently  $$\Omega \oplus T=_\ae \R^d.$$
As before, we let $\Tile(\Omega;\R^d)$ denote the space of all the solutions to the indeterminate tiling equation $\Omega\oplus X=_\ae \R^d$, with unknown $X\subset \R^d$. We study the structural behavior of this apace. 

\begin{example}[Hexagonal tiling]\label{ex:bee}
    Let $\Omega\subset \R^2$ be a centrally symmetric hexagon.  Then $\Tile(\Omega;\R^2)$ consists of all the cosets of the lattice (i.e., a discrete co-compact subgroup of $\R^2$) whose fundamental domain is $\Omega$. 
    This regular tiling appears in nature, e.g., in bee hives. See Figure \ref{fig:hex}.
\end{example}

\begin{figure}[ht]
    \centering
    \includegraphics[width=3.5cm]{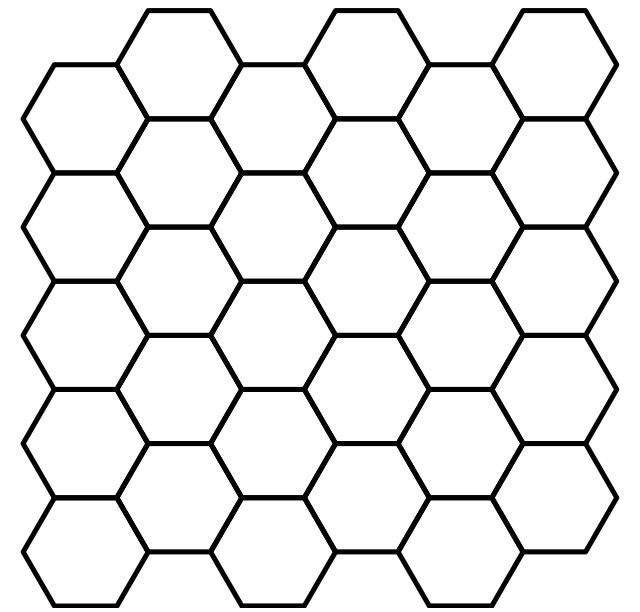}
    \hspace{4em}
    \includegraphics[width=3.1cm]{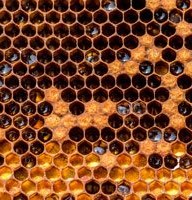}
    \caption{On the left--tiling by a symmetric hexagon, on the right--bee hive.}
    \label{fig:hex}
\end{figure}

\begin{example}[Cube tiling]\label{ex:cube}
    Let $\Omega= [0,1]^d$ be a $d$-dimensional unit cube. Then clearly $\Z^d\in \Tile(\Omega;\R^d)$ is a regular lattice tiling by $\Omega$. However, the space $\Tile(\Omega;\R^d)$ becomes extremely complicated in high dimensions.  In particular, Keller's conjecture, which asserts that any cube tiling contains two cube translates sharing an entire facet \cite{Keller}, turned out to be false in dimensions $d\geq 8$ \cite{LS,Mackey}. This is an instance of the phenomenon demonstrated in this article: the higher the dimension, the worse the structure of tilings.
\end{example}

Continuous tilings share many properties with discrete tilings, such as translation and reflection invariance, and the density property (see Section \ref{sec:discrete} above). 
However, a substantial difference is that there is no continuous analog of the dilation lemma \ref{lem:dilation}. Since this lemma plays a crucial role in our study of the structure of discrete tilings, we cannot adapt the discrete arguments to the continuous setup. As a result, new methods are needed to study the structure of continuous translational monotilings.

\subsection{Continuous periodic tiling conjecture.} 
We consider the following continuous analog of the discrete periodic tiling conjecture \ref{con:ptc}.

\begin{conjecture}[Continuous periodic tiling conjecture \cite{GS,LW}]\label{con:conptc}
    Let $\Omega\subset \R^d$ be a bounded measurable set of positive measures. Then the tiling equation $\Omega\oplus X=_\ae \R^d$ is not aperiodic. In other words, if the space $\Tile(\Omega;\R^d)$ is nonempty then it must contain a \emph{periodic} set $T$ (i.e., $T$ is invariant under translations by some lattice $\Lambda\subset \R^d$).  
\end{conjecture}

In the absence of a continuous dilation lemma and the presence of irrational points, it is significantly more difficult to establish positive results toward Conjecture \ref{con:conptc} than positive results toward Conjecture \ref{con:ptc}. 
The following partial supporting results are known:

\begin{itemize}
\item  Conjecture \ref{con:conptc} holds for convex tiles in all dimensions \cite{V,M}. In fact, it was shown that any convex set that tiles by translations admits a lattice tiling.
\item  In \cite{LW}, by proving the rationality of translational monotilings in $\R$, it was shown that the continuous periodic tiling conjecture in $\R$ can be reduced to the discrete periodic tiling conjecture in $\Z$. In turn, since Conjecture \ref{con:ptc} is known to hold in $\Z$, this proves Conjecture \ref{con:conptc} in $\R$. 
\item In $\R^2$, Conjecture \ref{con:conptc} is known to hold for topological disks \cite{bn,gbn,ken,err}. In was further shown that in this case all the tilings must be singly-periodic.
\item Recently, in \cite{dgm}, we also established Conjecture \ref{con:conptc} for rational polygonal sets, by modeling these tiles as discrete tiles and analyzing the internal structure of the earthquake components of their tilings.
\end{itemize}

\subsection{Continuous vs. discrete tilings.}
It is known, due to \cite{LW}, that Conjecture \ref{con:ptc} in $\Z$ implies Conjecture \ref{con:conptc} in $\R$. In dimensions $d\geq 2$, this reduction is not known. Thus, in particular, the continuous periodic tiling conjecture in $\R^2$ is still open, although the discrete periodic tiling conjecture in $\Z^2$ is known to hold. 
\begin{question}
    Does Conjecture \ref{con:conptc} hold in $\R^2$?
\end{question}

Discretizing continuous tiles turns out to be an extremely challenging task when the degree of irrational properties of the tiles increases.
Thus, for instance, discretizing polygonal sets with rational vertices is relatively simple, while discretizing polygonal sets assuming only rational slopes is significantly more subtle and discretizing polygonal sets with irrational slopes requires new tools, since in this case the slopes are not preserved under linear rational approximation and thus the combinatorial structure of the tiling need not be preserved.

On the other hand, in \cite[Section 2]{GT22}, by rigidly inflating a discrete tile into a continuous one, we showed that the converse reduction holds in all dimensions; that is, the continuous periodic tiling conjecture \ref{con:conptc} in $\R^d$ implies the discrete periodic tiling conjecture \ref{con:ptc} in $\Z^d$ for every $d\geq 1$. 
Therefore, for a sufficiently large $d$, our counterexample in $\Z^d$ yields a construction of a counterexample in $\R^d$ \cite{GT22}. Moreover, in \cite{GK}, by constructing a \emph{folded bridge} to connectify the latter counterexample in $\R^d$ and obtain a \emph{connected} counterexample in $\R^{d+2}$. This disproves a stronger version of Conjecture \ref{con:conptc} in which the tiles are required to be open and connected, in all sufficiently high dimensions. 
Indeed--the higher the dimension, the wilder the structure of the tilings.

\section*{Acknowledgments.}
The author is supported by NSF CAREER DMS-2441769 and the Alfred P. Sloan Fellowship. I am grateful to Bryna Kra and Terence Tao for their comments on an earlier version of this paper.

\bibliographystyle{plain}
\bibliography{references}
\end{document}